\documentclass[a4paper,12pt,reqno]{amsart}
\title{complete}

\title[complete intersection CY in Grassmannians]{Complete intersection Calabi--Yau manifolds\\
  with respect to homogeneous vector bundles\\ on 
Grassmannians}

\author[D.~Inoue]{Daisuke Inoue}
\address{
Graduate School of Mathematical Sciences,
The University of Tokyo,
3-8-1 Komaba,
Meguro-ku,
Tokyo,
153-8914,
Japan.}
\email{ino@ms.u-tokyo.ac.jp}

\author[A.~Ito]{Atsushi Ito}
\address{
Department of Mathematics,
Graduate School of Science,
Kyoto University,
Kyoto 606-8502,
Japan.
}
\email{aito@math.kyoto-u.ac.jp}

\author[M.~Miura]{Makoto Miura}
\address{
Korea Institute for Advanced Study,
85 Hoegiro,
Dongdaemun-gu,
Seoul,
130-722,
Republic of Korea.
}
\email{miura@kias.re.kr}

\date{}
\usepackage{mymacros}
\usepackage{hyperref}
\hypersetup{
    colorlinks=true
}
\usepackage[capitalise]{cleveref}
\crefname{thm}{Theorem}{Theorems}
\crefname{prop}{Proposition}{Propositions}
\crefname{lem}{Lemma}{Lemmas}
\crefname{subsubsection}{}{}
\crefname{equation}{}{}
\crefname{enumi}{}{}
\crefformat{equation}{(#1)}
\crefformat{enumi}{(#1)}
\theoremstyle{plain}
\newtheorem{thm}{Theorem}[section]
\newtheorem{cor}[thm]{Corollary}
\newtheorem{lem}[thm]{Lemma}
\newtheorem{prop}[thm]{Proposition}
\newtheorem{claim}[thm]{Claim}
\theoremstyle{definition}

\newtheorem{rmk}[thm]{Remark}
\newtheorem{example}[thm]{Example}
\theoremstyle{definition}
\newtheorem*{acknowledgements}{Acknowledgements}

\usepackage{colortbl}
\usepackage{multicol}

\newlist{assumptions}{enumerate}{10}
\setlist[assumptions]{label*=(A\arabic*)}
\crefname{assumptionsi}{}{}
\newlist{Qconditions}{enumerate}{10}
\setlist[Qconditions]{label*=(Q\arabic*)}
\crefname{Qconditionsi}{}{}
\newlist{propenum}{enumerate}{10}
\setlist[propenum]{label*=(\roman*)}
\crefname{propenumi}{}{}
\newcounter{exno}
\newenvironment{cy}
{\begin{tabular}{>{\refstepcounter{exno}}llllrrrrc}}
{\end{tabular}}
\crefname{exno}{No.\!}{No.\!}
\usepackage{colortbl}
\usepackage{multicol}

\begin{document}
\begin{abstract}
Based on the method by \cite{kuchle},
we give a procedure
to list up all complete intersection Calabi--Yau manifolds
with respect to direct sums of irreducible homogeneous vector bundles 
on Grassmannians for each dimension.
In particular,
we give a classification of such Calabi--Yau $3$-folds
and determine their topological invariants.
We also give alternative descriptions for some of them.
\end{abstract}
\maketitle


\thispagestyle{empty}
\section{Introduction}
In this paper, a smooth projective manifold
is called \textit{Calabi--Yau} if its first Chern class is trivial.
Calabi--Yau manifolds receive significant interest from 
mathematicians and physicists, not only for the 
classification of algebraic varieties 
but also for the relation with the string theory.
In dimensions greater than two,
there are many different deformation
classes of Calabi--Yau manifolds.
Many constructions of Calabi--Yau manifolds are known, for instance,
complete intersections of hypersurfaces in toric Fano varieties
(see e.g.\ \cite{CK} and references therein).
However, it is still an open problem whether or not the number of
deformation classes of Calabi--Yau 3-folds 
is finite.\\

Let $\scF$ be a globally generated 
homogeneous vector bundle on the Grassmannian $G(k,n)$
of $k$-dimensional subspaces in $\C^n$.
We denote by $Z_\scF\subset G(k,n)$ the zero locus 
of a general global section of $\scF$. 
From the Bertini-type theorem by \cite{mukai},  
$Z_\scF$ is a disjoint union of smooth submanifolds of $G(k,n)$
with $\codim Z_\scF=\rank \scF$ if it is not the empty set.
If that is the case,
we call $Z_\scF$ a \textit{complete intersection} 
with respect to $\scF$. 

K\"uchle has classified complete intersection Fano 4-folds
with respect to direct sums of irreducible homogeneous vector bundles on Grassmannians
\cite{kuchle}.
In this paper, we 
slightly generalize his result and obtain the 
list of all complete intersection Calabi--Yau $3$-folds with respect to
direct sums of irreducible homogeneous vector bundles on Grassmannians.

\begin{thm}
\label{mainthm_en}
Let $1< k < n-1$ be integers. 
A complete intersection Calabi--Yau $3$-fold $Z_\scF$ in $G(k,n)$
with respect to a direct sum of globally generated irreducible 
homogeneous vector bundles $\scF$
is one of the $3$-folds listed in \cref{tb:cy3} up
to natural identifications among the Grassmannians
explained in \cref{ss:natural}.
All except for \cref{row30} are irreducible, and
all except for \cref{row26} are Calabi--Yau $3$-folds in the 
strict sense, i.e.\ $h^1(\scO_{Z_{\scF}})=h^2(\scO_{Z_{\scF}})=0$.
\end{thm}

In \cref{tb:cy3}, the left-most column indicates the numbering.
The second column indicates the identifier of the corresponding Fano 4-folds of index one in \cite{kuchle}.
When such a Fano 4-fold $Z_{\scF'}$ exists, the vector bundle $\scF$ relates with $\scF'$ by
$\scF = \scF' \oplus \scO(1)$.
The third and fourth columns express the ambient Grassmannians and the vector bundles, respectively.
We use the symbols
$\scS$ and $\scQ$ for the universal subbundle and the universal 
quotient bundle on a Grassmannian $G(k,n)$, respectively.
Hence there exists an exact sequence on $G(k,n)$
\begin{align}
\label{eq:sq}
0\rightarrow \scS \rightarrow \scO^{\oplus n} \rightarrow \scQ \rightarrow 0.
\end{align}
The next four columns give the topological invariants of the Calabi--Yau 3-folds $Z_\scF$.
The right-most column indicates
alternative descriptions of the Calabi--Yau 3-folds, which we will explain below.

%
%

{\setlength{\textwidth}{200mm}
\newcommand*{\tabbox}[2][t]{%
    \vspace{0pt}\parbox[#1][3.7\baselineskip]{1cm}{\strut#2\strut}}
\def\arraystretch{1.1}
\small
\begin{table}[H]
\centering
\begin{cy}
\toprule
No.&\hspace{-7pt} \tiny{\cite{kuchle}} & $G(k,n)$ & $\scF$ & $H^3$ & $c_2.H$ & $c_3$ &
$h^{1,1}$  & Description \\ 
\midrule
\label{row1}\theexno & & $G(2,4)$ & $\scO(4)$ & $8$ & $56$ & $-176$ & $1$
& $\left(\P^5\right)_{2,4}$ \\
  \label{row2}\theexno &   (b2) & $G(2,5)$ & $\mathcal{O}(1) \oplus   \mathcal{O}(2)^{\oplus 2}$ & $20$ & $68$ & $-120$ & $1$ & $G(2,5)_{1,2^2}$\\
  \label{row3}\theexno & (b1) &&  $\mathcal{O}(1)^{\oplus 2} \oplus \mathcal{O}(3)$ & $15$ & $66$ &   $-150$ & $1$ & $G(2,5)_{1^2,3}$ \\
	\label{row4}\theexno &&& $\scS^{*}(1) \oplus \mathcal{O}(2)$ &   $24$ & $72$ & $-116$ & $1$ & $OG(5,10)_{1^6, 2}$  \\
  \label{row5}\theexno &&& $\wedge^2 \scQ(1)$ &   $25$ & $70$ & $-100$ & $1$ & $G(2,5) \cap G(2,5) \leadsto$  \\
  \label{row6}\theexno & (b6) & $G(2,6)$ & $\mathcal{O}(1)^{\oplus 4}   \oplus \mathcal{O}(2)$ & $28$ & $76$ & $-116$ & $1$ &   $G(2,6)_{1^4,2}$ \\
	\label{row7}\theexno & (b5) && $\scS^{*}(1) \oplus \mathcal{O}(1)^{\oplus 3}$ &   $33$ & $78$ & $-102$ & $1$ & $\Sigma_{1^9} $  for a Schubert variety $\Sigma$ \\
  \label{row8}\theexno & (b4) && $\Sym^2\scS^{*}   \oplus \mathcal{O}(1) \oplus \mathcal{O}(2)$  & $40$ & $88$ & $-128$ & $2$ & $(\P^3\times \P^3)_{1^2,2}$    \\
	\label{row9}\theexno && &  $\Sym^2\scS^{*} \oplus \scS^{*}(1)$ & $48$ & $84$ & $-92$ & $2$ & \cref{sec_description4}\\
	\label{row10}\theexno & (b3) && $\scQ(1) \oplus \mathcal{O}(1)$ &   $42$ & $84$ & $-98$ & $1$ & $G(2,7)_{1^7}$\\
  \label{row11}\theexno &&&   $\wedge^3\scQ \oplus \mathcal{O}(3)$ & $18$ & $72$ & $-162$ & $2$ &  $(\P^2\times \P^2)_{3}$\\
  \label{row12}\theexno &   (b7) & $G(2,7)$ &   $\mathcal{O}(1)^{\oplus 7}$ & $42$ & $84$ & $-98$ & $1$ & $G(2,7)_{1^7}$ \\
  \label{row13}\theexno & (b8) &&   $\Sym^2\scS^{*} \oplus \mathcal{O}(1)^{\oplus 4}$ & $56$ & $92$ & $-92$  & $1$ & $OG(2,7)_{1^4}$ \\
  \label{row14}\theexno & (b9) &&   $(\Sym^2\scS^{*})^{\oplus 2} \oplus \mathcal{O}(1)$ & $80$ & $80$ &  $-32$ & $8$ & \cite{cas} \\
  \label{row15}\theexno & (b10) &&   $\wedge^4\scQ \oplus \mathcal{O}(1) \oplus \mathcal{O}(2)$ & $36$ &  $84$ & $-120$ & $1$ & $\left(G_2/P_1\right)_{1,2}$, \cref{sec_description4}\\
	\label{row16}\theexno &&& $\scS^{*}(1) \oplus \wedge^4\scQ$ & $42$ & $84$ & $-98$ & $1$ & 
	$G(2,7)_{1^7}  \leadsto$, \cref{sec_description4}\\
  \label{row17}\theexno & (b11) & $G(2,8)$ &   $\wedge^5\scQ \oplus \mathcal{O}(1)^{\oplus 3}$ & $57$ & $90$ & $-84$ &  $1$ & \cref{sec_description4} \\
  \label{row18}\theexno &&&   $\Sym^2\scS^{*} \oplus \wedge^5\scQ$ & $72$ & $96$ & $-72$ & $1$ &   \cref{sec_description4} \\
  \label{row19}\theexno &   (c1) & $G(3,6)$ &   $\mathcal{O}(1)^{\oplus 6}$ & $42$ & $84$ & $-96$ & $1$ & $G(3,6)_{1^6}$ \\
  \label{row20}\theexno & (c2) &&   $\wedge^2 \scS^{*} \oplus \mathcal{O}(1)^{\oplus 2} \oplus \mathcal{O}(2)$  & $32$ & $80$ & $-116$ & $1$ & $LG(3,6)_{1^2,2}$ \\
  \label{row21}\theexno &&& $\scS^{*}(1) \oplus \wedge^2 \scS^{*}$ &   $42$ & $84$ & $-96$ & $1$ & $G(3,6)_{1^6}  \leadsto$ \\
  \label{row22}\theexno & (c4) & $G(3,7)$ &   $\Sym^2 \scS^{*} \oplus \mathcal{O}(1)^{\oplus 3}$ & $128$ & $128$ &  $-128$ & $1$ & $\left(\P^7\right)_{2^4}$ \\
  \label{row23}\theexno & (c6) &&   $(\wedge^2 \scS^{*})^{\oplus 2} \oplus \mathcal{O}(1)^{\oplus 3}$ &  $61$ & $94$ & $-86$ & $1$ & \\
  \label{row24}\theexno & (c3) &&   $(\wedge^3 \scQ)^{\oplus 2} \oplus \mathcal{O}(1)$ & $72$ & $96$ &  $-74$ & $1$ & \\
  \label{row25}\theexno & (c5) &&   $\wedge^2 \scS^{*} \oplus \wedge^3 \scQ \oplus\mathcal{O}(1)^{\oplus 2}$   & $66$ & $96$ & $-84$ & $1$ & \cite{Ku2} \\
  \label{row26}\theexno && $G(3,8)$ &   $(\Sym^2 \scS^{*})^{\oplus 2}$ & $384$ & $0$ & $0$ & $9$ & Jacobian of a curve of genus $3$\\
  \label{row27}\theexno &&&   $\Sym^2 \scS^{*} \oplus (\wedge^2 \scS^{*})^{\oplus 2}$ &  $176$ & $128$ & $-64$ & $2$ & \\
  \label{row28}\theexno &&&   $(\wedge^2 \scS^{*})^{\oplus 4}$ & $92$ & $104$ & $-64$ & $1$ & \\
  \label{row29}\theexno & (c7) &&   $\wedge^3 \scQ \oplus \mathcal{O}(1)^{\oplus 2}$ & $102$ & $108$ &  $-84$ & $2$ & crepant resolution of $Z' \subset \P^5$\\
  \label{row30}\theexno & (d1) & $G(4,8)$ &   $\Sym^2 \scS^{*} \oplus \mathcal{O}(1)^{\oplus 3}$ & $256$ & $256$ &  $-256$ & $2$ &   $\left(\P^7 \right)_{2^4}\sqcup \left(\P^7 \right)_{2^4}$ \\
  \label{row31}\theexno &&&   $(\wedge^2 \scS^{*})^{\oplus 2} \oplus \mathcal{O}(2)$ &   $48$ & $96$ & $-128$ & $4$ &   $\left(\prod^4 \P^1\right)_2 $ \\
  \label{row32}\theexno & (d2) & $G(4,9)$ &   $\Sym^2 \scS^* \oplus \wedge^2 \scS^* \oplus \mathcal{O}(1)$ &  $384$ & $192$ & $-128$ & $4$ &   $\left(\prod^4 \P^1\right)_2$ \\
  \label{row33}\theexno & (d3) & $G(5,10)$ &   $(\wedge^2 \scS^{*})^{\oplus 2} \oplus \mathcal{O}(1)^{\oplus 2}$ &  $120$ & $180$ & $-220$ & $5$ &   $\left(\prod^5 \P^1\right)_{1^2}$ \\
\bottomrule
\end{cy}
\caption{Complete intersection Calabi--Yau $3$-folds in
  Grassmannians}
\label{tb:cy3}
\end{table}
\normalsize
}

In \cite{MR3848432},
V. Benedetti independently gives a similar classification
in the context 
of classifying hyperk\"ahler 4-folds. 
He classifies complete intersection Calabi--Yau $d$-folds for $2\le d \le 4$
with respect to
direct sums of globally generated irreducible homogeneous vector bundles on
not only ordinary Grassmannians but also symplectic and 
orthogonal Grassmannians.
In his list,
there are examples of Calabi--Yau 3-folds other than those in \cref{tb:cy3},
 such as Borcea's 
examples \cite{Bor}, i.e.\ \cite[Table 10, (ow15)--(ow18)]{MR3848432}.
On the other hand, 
we also show the finiteness of such Calabi--Yau $d$-folds in ordinary Grassmannians for each $d$
(\cref{mainprop}) and
study alternative descriptions of the obtained Calabi--Yau $3$-folds.

In \cite{Ku},
Kuznetsov gave alternative descriptions of the Fano $4$-folds in
\cite{kuchle}  with Picard number greater than $1$, i.e.\ (b4), (b9), (c7), (d3) in K\"{u}chle's list.
In the paper,
Kuznetsov asked whether $Z_{\scQ (1)} \subset G(2,6)$ and $Z_{\calo(1)^{\oplus 6}} \subset G(2,7)$ are deformation equivalent or not
since they have the same collections of discrete invariants.
Soon after a preliminary version of \cite{Ku} appeared,
Manivel gave an affirmative answer in \cite{Ma}.
In fact, Manivel showed that these two types of varieties 
are the same up to projective equivalence.

Following \cite{Ku} and \cite{Ma},
we also study alternative descriptions of most of the Calabi--Yau 3-folds
in \cref{tb:cy3}
in \cref{sec_description,sec_description2,sec_description4}.
We explain the notation in the right-most column of \cref{tb:cy3} briefly.

We use the notation $X_{d_1,\dots,d_r}$,
which means a complete intersection 
of $r$ general hypersurfaces of degree $d_1 ,\dots , d_r$ in $X$.
If $X$ is one of Grassmannians $G(k,n)$,
(irreducible components of)
orthogonal Grassmannians $OG(k,n)$, Lagrangian Grassmannians
$LG(k,2k)$, a $G_2$-Grassmannian $G_2/P_1$, or $\Sigma$ in \cref{row7}, which is a Schubert variety in the Cayley plane $\mathbb{O}\P^2$,
the degrees are defined with respect to the unique ample
generator of the Picard group of the homogeneous space. 
For the case of  $X=\prod^l \P^s$, the degrees of 
the hypersurfaces appearing here
are defined with respect to $\scO_X(1,\dots ,1)$.
The symbol $\star \leadsto$ in \cref{row5,row16,row21} means that the Calabi--Yau 3-folds in \cref{row5,row16,row21}
are specializations of Calabi--Yau 3-folds $\star$.
The variety $Z' \subset \P^5$ in \cref{row29} is a general complete intersection of two cubic hypersurfaces containing the Veronese surface $\P^2 \hookrightarrow \P^5$.

For instance,
\cref{row4,row7} are described as complete intersections of hypersurfaces in 
Schubert varieties in \cref{rmk_galkin}.
Similarly, 
\cref{row5} is a specialization of a complete intersection $G(2,5)\cap G(2,5)$ of two
Grassmannians by \cref{prop_No4}.
\cref{row16,row21} are also specializations of complete intersections of 
hyperplanes in the same Grassmannians by \cref{prop_specialize_of_(1)^n}.
Although experts might have known such descriptions,
we give proofs
since we could not find them written explicitly in the literature.
Other descriptions not mentioned here are summarized in \cref{sec_description4}.
\\

\vspace{2mm}
This paper is organized as follows.
In \cref{sec:finite}, we show \cref{mainprop},
which states that there are at most finitely many 
families of complete intersection
Calabi--Yau $d$-folds with respect to direct sums of 
irreducible homogeneous vector bundles on
Grassmannians for a fixed dimension $d>0$.
The proof gives a concrete procedure to classify such $d$-folds.
In \cref{sec:CY3}, 
we exemplify the classification of such Calabi--Yau $3$-folds,
and show
\cref{mainthm_en}.
In  \cref{sec_description,sec_description2,sec_description4},
we study alternative descriptions of Calabi--Yau $3$-folds
in \cref{tb:cy3}.
Throughout this paper,
we work over the complex number field $\C$.

\begin{acknowledgements}
The authors thank Professor Shinobu Hosono and Doctor Fumihiko Sanda 
for valuable discussions 
at weekly seminars and for various useful comments.
The authors also thank the anonymous referees
for providing a number of valuable comments and suggesting to simplify the arguments on the
alternative description of \cref{row15,row17,row18} in \cref{sec_description4}.
A.~I.~was supported by the Grant-in-Aid for JSPS fellows, No.\ 26--1881.
A part of this work was done when M.~M.~was supported by Frontiers of
Mathematical Sciences and Physics at University of Tokyo. 
M.~M.~was also supported by Korea Institute for Advanced Study.
\end{acknowledgements}

\section{Finiteness of complete intersection Calabi--Yau manifolds}
\label{sec:finite}

In this section, we give a procedure to list up
all direct sums of irreducible homogeneous vector bundles $\scF$ on Grassmannians 
such that $Z_\scF$ is a $d$-dimensional Calabi--Yau manifolds
for a fixed $d >0$.\\


Recall that $\scS$ and $\scQ$ are the universal subbundle and the universal 
quotient bundle on a Grassmannian $G(k,n)$, respectively.
We denote by $S_\lambda \scS^*$ (resp.\ $S_\mu \scQ$)
the globally generated homogeneous vector bundle corresponding to
the Schur module $S_\lambda \C^k$ (resp.\ $S_\mu \C^{n-k}$) with respect
to a Young diagram $\lambda=(\lambda_1\ge \dots \ge \lambda_k\ge 0)$
of $GL(k)$
(resp.\ $\mu=(\mu_1\ge\dots \ge \mu_{n-k}\ge 0)$ of $GL(n-k)$).

Any irreducible globally generated 
homogeneous vector bundle $\scE$ on $G(k,n)$ can be written as 
\begin{align}
  \label{generalform_en}
\scE=S_\lambda \scS^* \otimes S_\mu \scQ \otimes \scO(p)
\end{align}
for some $\lambda,\mu$ and $p \geq 0$ with $\lambda_k=0$ and 
$\mu_{n-k}=0$.
We note that 
$S_\lambda \scS^*(p)=S_\lambda \scS^*\otimes \scO(p)$ is isomorphic to 
$S_{\lambda'} \scS^*$ 
for 
$\lambda'=(\lambda_1'\ge \dots \ge \lambda'_k\ge 0)$ with
$\lambda_j'=\lambda_j+p$, and so is $S_\mu \scQ(p)= S_\mu \scQ \otimes \scO(p)$.\\

Let $\scF$ be 
a direct sum of globally generated irreducible homogeneous vector bundles on 
a Grassmannian $G(k,n)$
such that 
the zero locus $Z_\scF$ of a general global section
of $\scF$ has the properties that $\dim Z_\scF=d$ and $  c_1(Z_\scF)=0$.
By the adjunction formula for $Z_\scF \subset G(k,n)$,
it holds that
\begin{align}
  \label{conditions_en}
  \rank\scF = k(n-k)-d
  \quad \text{ and } \quad
  c_1(\scF) = n.
\end{align}

The conditions \cref{conditions_en} impose several
restrictions on an irreducible component
$\scE\subset \scF$.
One of the restrictions is the following. 
\begin{lem}[{\cite[Corollary 3.5 (a)]{kuchle}}]
  \label{threekinds_en}
  Let 
  $\scE  = S_\lambda \scS^*\otimes S_\mu \scQ \otimes \scO(p)$ 
   with $\lambda_1 \ge \dots \ge \lambda_k = 0$ and $\mu_1 \ge \dots
  \ge \mu_{n-k} = 0$
  be an irreducible component of $\scF$.
  Then  either $\lambda = 0$ or $\mu = 0$ holds, so that
  $\scE = S_\lambda \scS^* (p)$ or $\scE = S_\mu \scQ (p)$. 
\end{lem}
\begin{proof}
Otherwise, it contradicts $d>0$ since $\rank \scF \geq \rank \scE \ge 
\rank \scS \cdot \rank \scQ =  k(n-k)$.
\end{proof}

Let $\scF$ be a direct sum of globally generated irreducible  homogeneous vector bundles on $G(k,n)$
such that $Z_{\scF} $ is a Calabi--Yau $d$-fold.
By \cref{threekinds_en},
such $\scF$ must be decomposed as 
  \begin{align}
  \label{eq_decomp}
  \scF=\scF^\scS \oplus \scF^\scQ \oplus \scF^{\mathrm{line}},
\end{align}
where $\scF^\scS$, $\scF^\scQ$ and $\scF^{\mathrm{line}}$ are direct sums of 
irreducible homogeneous vector bundles of type $S_\lambda \scS^*$,
$S_\mu \scQ$, and $\scO(p)$ with $\lambda_1\ne \lambda_k \geq0$, $\mu_1\ne \mu_{n-k}\geq 0$, and $p> 0$, respectively.

We consider the following conditions \cref{ass:1}, \dots,
\cref{ass:4}.
\begin{assumptions}
  \setlength{\itemsep}{2pt}
  \item \label{ass:1} $k\ge2$,
  \item \label{ass:2} $n\ge 2k$,
  \item \label{ass:3} $\scS^*\not \subset \scF^\scS$ and $\scQ\not \subset \scF^\scQ$,
  \item \label{ass:4} 
	$\Sym^2 \scQ \not
	\subset \scF^\scQ$ and $\wedge^2 \scQ \not \subset \scF^\scQ$ 
	for $n>2k$,
\end{assumptions}
where $\scE \not \subset \scE'$ means that $\scE$ is not contained in $\scE'$ 
as an irreducible component.

\vspace{1mm}
In fact, we may assume  \cref{ass:1}, \dots,
\cref{ass:4} in addition to \cref{conditions_en} for the classification in \cref{mainthm_en}
as we explain as follows.

If $k=1$, i.e.\ if $\scF$ is on
a projective space $\P^{n-1}$,
we easily see that $\scF^\scS= \scF^\scQ = 0$ 
because $\scS^*=\scO(1), \rank \scQ =n-1$, and $d>0$.
Hence we may assume \cref{ass:1}
unless we consider complete intersections of 
hypersurfaces in projective spaces.

Set $l=n-k$.
Under the natural isomorphism $G(k,n)\simeq G(l,n)$,
the homogeneous vector bundles $S_\lambda \scS^*$ and $S_\mu \scQ$ on $G(k,n)$
are transformed to $S_\lambda \scQ$ and $S_\mu \scS^*$ on $G(l,n)$, respectively.
Hence we assume \cref{ass:2} to avoid the duplication.

There is another natural isomorphism between 
$Z_{\scS^*}\subset G(k,n)$ and $G(k,n-1)$.
Under the isomorphism, the restrictions of homogeneous vector bundles 
$\scS^*$ and  $\scQ$ on $Z_{\scS^*}$ correspond to
$\scS^*$ and $\scQ\oplus \scO$ on $G(k,n-1)$, respectively.
A similar property holds for
$Z_\scQ\subset G(k,n)$ and $G(k-1,n-1)$.
Hence we assume \cref{ass:3} as in \cite[Lemma 3.2 (ii)]{kuchle}.


Finally, by considering possible dimensions of isotropic
subspaces in $\C^n$ equipped with a symmetric or a skew-symmetric
form of maximal rank,
we see that $Z_{\Sym^2 \scQ}=\emptyset$ for $n>2k$ and
$Z_{\wedge^2 \scQ}=\emptyset$ for $n>2k+1$.
Furthermore, $Z_{\wedge^2 \scQ}\subset G(k,2k+1)$ is 
isomorphic  to $Z_{\wedge^2 \scQ}\subset G(k,2k)$,
under which the restrictions of homogeneous vector bundles on 
$Z_{\wedge^2 \scQ}\subset G(k,2k+1)$
are transformed to the restrictions of other kinds of homogeneous vector bundles on
$Z_{\wedge^2 \scQ}\subset G(k,2k)$.
Hence we also assume $\wedge^2 \scQ \not \subset \scF$ for 
$n=2k+1$. Therefore, we assume \cref{ass:4} as in \cite[Lemma 3.2 (iii)]{kuchle}.

\vspace{2mm}
In this section,
we show the following proposition.

\begin{prop}
  \label{mainprop}
  Consider a globally generated homogeneous vector bundle $\scF =  \scF^\scS \oplus \scF^\scQ \oplus \scF^{\mathrm{line}}$ on $G(k,n)$, which is decomposed as in \cref{eq_decomp}.
  
For a fixed positive integer $d >0$,
there are at most finitely many choices of positive integers $k ,n$,
and bundles $\scF$ on $G(k,n)$
such that \cref{conditions_en} and
the 
assumptions \cref{ass:1}, \dots,
\cref{ass:4} are satisfied.
\end{prop}




In conclusion,
we have the following corollary of \cref{mainprop}.
\begin{cor}\label{cor_finiteness}
For a fixed positive integer $d >0$,
there are at most finitely many families of complete intersection Calabi--Yau $d$-folds in Grassmannians 
with respect to direct sums of globally generated irreducible
homogeneous vector bundles.
\end{cor}

\vspace{2mm}

In the rest of this section,
we give a proof of \cref{mainprop}.
To show \cref{mainprop},
we may assume one more condition in the case of $n=2k$.

\begin{lem}\label{lem_F^Q=0}
To prove \cref{mainprop},
it suffices to show the finiteness of the choices of $k,n$, and $\scF$
under the additional condition
\begin{assumptions}
  \setlength{\itemsep}{2pt}
  \setcounter{assumptionsi}{4}
  \item \label{ass:5} $\scF^\scQ=0$ if $n=2k$.
\end{assumptions}
\end{lem}

\begin{proof}
Assume $n=2k$ and take $ \scF=\scF^\scS \oplus \scF^\scQ \oplus
\scF^{\text{line}}$ that satisfies \cref{conditions_en} and
\cref{ass:1},\dots,\cref{ass:4}.
Write $\scF^\scQ=\bigoplus_{\mu} ({S_{\mu}\scQ })^{\oplus a_{\mu} }$ with $a_{\mu} \geq 0$.
Since $\rank \scS^* =\rank \scQ$ and $c_1(\scS^*)=c_1(\scQ)$,
$\scF' := \scF'^\scS  \oplus \scF^{\text{line}} $ also satisfies \cref{conditions_en} and \cref{ass:1},\dots,\cref{ass:4},
where $\scF'^\scS := \scF^\scS \oplus \bigoplus_{\mu} ({S_{\mu}\scS^* })^{\oplus a_{\mu} }$.

Then $\scF$ is obtained from $\scF'  $ by replacing some components $S_{\lambda} \scS^* $ of $\scF' $ by $S_{\lambda} \scQ  $.
Since we have at most finitely many such $\scF$ from a fixed $\scF'$,
it suffices to show the finiteness of $\scF'$, which satisfies the additional condition \cref{ass:5}.
\end{proof}


\begin{example}\label{eg_A5}
Contrary to \cref{ass:2}--\cref{ass:4}, \cref{ass:5} is not a condition to avoid duplication.
Consider $\scF_1= (\wedge^3 \scS^*)^{\oplus 2} \oplus \scO(1)^{\oplus 2}$ on $G(4,8)$.
Then $\rank {\scF_1}=10$, $c_1(\scF_1)=8$, and hence $Z_{\scF_1} $ is a Calabi--Yau $6$-fold.
Replacing one of the irreducible component $\wedge^3 \scS^* \subset \scF$ by $\wedge^3 \scQ$,
we obtain $\scF_2 =  \wedge^3 \scS^* \oplus \wedge^3 \scQ \oplus \scO(1)^{\oplus 2}$,
from which we have another Calabi--Yau $6$-fold $Z_{\scF_2} $.
Since the Euler numbers of these varieties are computed as
$\chi \left(Z_{\scF_1}\right)=14148 \text{ and } 
\chi \left( Z_{\scF_2} \right) =14328$,
we see that $Z_{\scF_1}$ and $Z_{\scF_2}$ are not isomorphic.
\end{example}

Throughout Subsections \ref{sec:Q} and \ref{sec:S},
$\scF=\scF^\scS \oplus \scF^\scQ \oplus \scF^{\text{line}}$
is a vector bundle as in \cref{mainprop}
satisfying \cref{conditions_en} and \cref{ass:1}--\cref{ass:5}.

\subsection{The case of \texorpdfstring{$\scF^\scQ\ne 
  (\wedge^{l-1}\scQ)^{\oplus A}\oplus
  \left(\scQ(1)\right)^{\oplus B}$}{}}
\label{sec:Q}
We begin with the case $\scF^\scQ\ne0$.
In this case, we have $n>2k$ by
\cref{ass:2,ass:5}.
Let us write $l:=n-k>k$.

 \begin{lem}
  \label{lem_remainS}
  Let $\scE$ be an irreducible component of $\scF$. 
  If $0\le c_1(\scE) \le k-1$,
   then $\scE$ is one of 
  $\scO(p)$ for some $0 < p \le k-1$, $\wedge^2 \scS^*$, and $\wedge^{k-1}\scS^*$.
 \end{lem}

 \begin{proof}

	If $\scE \subset \scF^{\text{line}}$, it is obvious that
	$\scE = \scO(p)$ for some $0 < p \le k-1$.
	From \cref{threekinds_en}, we may assume $\scE =S_\lambda \scS^*$ or $S_\mu \scQ$
	with $\lambda_1 \neq \lambda_k \geq 0$ or $\mu_1 \neq \mu_l \geq 0$. 

	Let $\scE=S_\lambda \scS^*$ first.
	We have formulas
  \begin{align}\label{eq_rank_e_S}
	 &\rank S_\lambda \scS^*=
	 \prod_{1\le i<j\le k}\frac{j-i+\lambda_i-\lambda_j}{j-i},\\
   &\label{eq_c1_e_S} c_1(S_\lambda \scS^*)=\frac{\abs{\lambda}}{k}\rank S_\lambda \scS^*,
  \end{align}
	where $\abs{\lambda} =\lambda_1+\dots +\lambda_k$.
	By \cite[Table 1]{MR0267040}, $\rank \scE \ge k^2-1$ holds 
	for all $\scE$ except
	$\scS^*(p)$, $\wedge^2 \scS^*(p)$, $\Sym^2 \scS^*(p)$, $\wedge^{k-1}\scS^*(p)$, 
	$S_{(2,\dots, 2,0)} \scS^*(p)$, $\wedge^{k-2}\scS^*(p)$, $\wedge^3 \scS^*(p)$ $(k=6,7,8)$
	and $\wedge^{k-3} \scS^*(p)$  $(k=6,7,8)$ 
  for $p\ge 0 $.
	In this case, it turns $\abs{\lambda}\ge 3$ and
	\begin{align}
					c_1(\scE) = 
					\frac{\abs{\lambda}}{k} \rank
                                        \scE \ge  \frac{\abs{\lambda}}{k}(k^2-1) \ge 3k -\frac{3}{k} 
					>k-1.
	\end{align}
	Assume $\cale$ is one of the above exceptions $\scS^*(p), \wedge^2 \scS^*(p), \dots, \wedge^{k-3} \scS^*(p)$  $(k=6,7,8)$.
	We note $\cale \neq \scS^*$ by \cref{ass:3}.
	Then we can check $c_1(\cale) \geq   k-1$ and that the equality holds only for $\cale = \wedge^2 \scS^* $ or $\wedge^{k-1}\scS^*$. 
	Hence $\cale = \wedge^2 \scS^* $ or $\wedge^{k-1}\scS^*$ by $c_1(\cale) \leq k-1$.
	
			In the case of $\scE=S_\mu \scQ$, 
			we use the formulas
			  \begin{align}\label{eq_rank_e}
	 &\rank S_\mu \scQ=
	 \prod_{1\le i<j\le l}\frac{j-i+\mu_i-\mu_j}{j-i},\\
   &\label{eq_c1_e} c_1(S_\mu \scQ)=\frac{\abs{\mu}}{l}\rank S_\mu \scQ.
  \end{align}
	By exactly the same calculation by replacing
	 $\lambda, k$ with $\mu, l$, respectively,
	 we have $ c_1(\cale) \geq l-1 > k-1$.
	 Hence there is no $\scE=S_\mu \scQ$ with $c_1(\cale) \leq k-1$.
 \end{proof}
\begin{lem}
  \label{prop:Q}
  Assume that there exists an irreducible component
  $\scE$ of $\scF^\scQ$ which is not 
  $\wedge^{l-1}\scQ$ nor $\scQ(1)$.
  Then $\scF$ is one of the following:
\def\arraystretch{1.2}
\begin{table}[H]
\centering
\begin{tabular}{ccccc}
\hline
 & $G(k,n)$ & $\scF$ & $d=\dim Z_\scF$ \\
\hline
($\alpha$1) & $G(3,8)$  & $\wedge^3 \scQ\oplus \scO(2)$ & $4$ \\
($\alpha$2) &  & $\wedge^3 \scQ\oplus \scO(1)^{\oplus 2}$ & $3$ \\
($\alpha$3) &  & $\wedge^2 \scS^*\oplus  \wedge^3 \scQ$ & $2$ \\
($\alpha$4) & $G(4,9)$  & $\wedge^3 \scQ\oplus \scO(3)$ & $9$ \\
($\alpha$5) &  & $\wedge^3 \scQ\oplus \scO(2)\oplus \scO(1)$ & $8$ \\
($\alpha$6) &  & $\wedge^3 \scQ\oplus \scO(1)^{\oplus3}$ & $7$ \\
($\alpha$7) &  & $\wedge^2 \scS^*\oplus  \wedge^3 \scQ$ & $4$ \\
($\alpha$8) &  & $\wedge^3 \scS^*\oplus  \wedge^3 \scQ$ & $6$ \\
($\alpha$9) & $G(4,10)$ & $\wedge^3 \scQ$ & $4$ \\
($\alpha$10) &   & $\wedge^4 \scQ$ & $9$ \\
($\alpha$11) & $G(5,11)$ & $\wedge^3 \scQ\oplus \scO(1)$ & $9$ \\
($\alpha$12) &  & $\wedge^4 \scQ\oplus \scO(1)$ & $14$ \\
($\beta k$) & $G(k,2k+1)$ & $\wedge^k\scQ(1)$ & $k^2-1$ \\
\hline
\end{tabular}
\end{table}
\end{lem}
\begin{proof}
Let $\scE = S_\mu \scQ(p)$ ($\mu_l=0$, $p\ge 0$) be an irreducible component of
$\scF^\scQ$ and
\begin{align}
 \delta&:=kl - \rank \scE = \dim Z_\scE,\\
 \iota&:=n -c_1(\scE)=c_1(Z_\scE).
\end{align}

In the following, we list up all $\scE$ that satisfy
the conditions 
$\delta>0$ and $\iota\ge 0$ by using the formulas
\cref{eq_rank_e}, \cref{eq_c1_e}.


First, we consider any $\scE$ with $\mu \ne 
(1, 0, \dots, 0)$, $(1,1,0,\dots, 0)$, $(2,0,\dots, 0)$, $(1,\dots, 1, 0)$,
$(2,\dots 2, 0)$, $(1,\dots, 1, 0, 0)$, 
nor $(1,1,1,0,\dots,0)$, $(1,\dots,1,0,0,0)$ for $l=6,7,8$.
In this case, 
it holds $\rank \scE \ge l^2 -1$ again by \cite[Table 1]{MR0267040}.
Then we have $\delta\le kl-(l^2-1)<0$ (since $l>k$).
This contradicts the condition $\delta>0$.

Let $\mu = (1,1,1,0,\dots,0)$ or $(1,\dots,1,0,0,0)$ for $l=6,7,8$.
In this case, all the possible irreducible bundles $\scE$ 
with $\delta>0$ can be listed in
the following.
\def\arraystretch{1.2}
\begin{align*}
\begin{array}{ccrr}
\scE & G(k,n) & \delta=\dim Z_\scE & \iota=c_1(Z_\scE)\\
\midrule
\wedge^3 \scQ(p) & G(4,10) & 4 & -20p\\ 
& G(5,11) & 10 & 1-20p\\ 
& G(6,13) & 7 & -2-35p \\ 
\wedge^4 \scQ(p) & G(6,13) & 7 & -7-35p\\ 
\end{array}
\end{align*}
The second condition $\iota \ge 0$
gives two examples of $\scF$,
($\alpha$9) and ($\alpha$11) from the first and second cases in this list with $p=0$.

Let $\mu =(1, 0, \dots, 0)$, i.e.\ $\scE=\scQ(p)$.
We may assume $p\ge 2$ by the assumptions \cref{ass:3} and $\scE \ne \scQ(1)$.
Then we have a contradiction
\begin{align}
n \ge c_1(\scE)=pl+1\ge 2l+1 >n
\end{align}
since $p\ge 2$ and $l>\frac{n}{2}$.

In the case of $\mu = (1,1,0,\dots, 0)$, $(1,\dots, 1, 0)$,
(i.e.\ $\scE=\wedge^2 \scQ(p)$, $\wedge^{l-1}\scQ(p)$, respectively),
we may assume $p\ge 1$ by
the assumptions \cref{ass:4} and $\scE \ne \wedge^{l-1}\scQ$,
respectively.
In the former case, it holds that
\begin{align}
  n \ge c_1(\scE)=l-1+p\binom{l}{2}
  \ge l-1+\binom{k+1}{2}
  \ge l+k =n
\end{align}
for any $k\ge 2$. The equalities hold simultaneously only for $k=2$, $l=k+1=3$ and $p=1$,
i.e.\ the case ($\beta 2$).
In the latter case, we also observe
\begin{align}
  n \ge c_1(\scE)=l-1+p l\ge 2l-1 \ge n,
\end{align}
since $p\ge 1$ and $l >\frac{n}{2}$.
The equalities hold simultaneously only for ($\beta k$)'s.

Let $\mu = (2,0,\dots, 0)$,
i.e.\ $\scE = \Sym^2 \scQ(p)$.
Note that the case of $p = 0$ is 
already excluded by the assumption \cref{ass:4}.
Then it holds by \cref{eq_c1_e} that 
\begin{align}
 \label{eq_c1_ex}
 c_1(\scE) = \binom{l+1}{2}\left(\frac{2}{l}+p\right) \ge
   \binom{l+1}{2}\left(\frac{2}{l}+1\right)=\frac{1}{2}(l+1)(l+2).
\end{align} 
This leads to a contradiction
 $\iota \le n-\frac{1}{2}(l+1)(l+2)<0$,
 since for $l \ge 2$ one has
\begin{align}
 \frac{1}{2}(l+1)(l+2) \ge 2(l+1) >2l > n.
\end{align}

As for $\mu = (2,\dots, 2, 0)$, 
we have
\begin{align}
 c_1(\scE)=\binom{l+1}{2}\left( \frac{2(l-1)}{l}+p\right)\ge 
 \binom{l+1}{2}\left(\frac{2(l-1)}{l}\right)=l^2-1.
\end{align}
It also leads us to
$\iota\le n-(l^2-1)<0$,  since for $l\ge k+1 \ge 3$ one has
 \begin{align}
  (l+1)(l-1) \ge 4(l-1) = 2l + 2(l-2) > n.
 \end{align}

Finally, we consider the case of $\mu = (1,\dots, 1, 0, 0)$,
i.e.\ $\scE=\wedge^{l-2}\scQ(p)$ with $p\ge 0$.
We may assume $l \geq 5$ since the cases of $\mu =(1,0,0)$ and  $(1,1,0,0)$ are already treated.
If $p\ge 1$, the condition $\iota =n-c_1(\scE) \ge 0$ gives
a contradiction
\begin{align}
  n\ge c_1(\scE)=
 \binom{l-1}{2}+p\binom{l}{2} \ge
 (l-1)^2> n
\end{align}
since $n\le 2l-1$ and $l\ge 5$. 
Let $p=0$.
 For $l \ge 7$, one has a contradiction
 \begin{align}
  n \ge c_1(\scE) = \frac{1}{2}(l-1)(l-2) \ge 3(l-2) = (2l-1)+(l-5) > n.
 \end{align}
This gives a bound $l \le 6$, i.e.\ $l=5$ or $6$.
Together with $l\ge k+1$, one can list all possible $(k,l)$
by using another condition 
\begin{align}
  0<\delta=
  kl- \rank \scE  =kl-\binom{l}{2}.
\end{align}
Namely, 
the possible irreducible vector bundle $\scE=\wedge^{l-2}\scQ$
with $\delta>0$ and $\iota\ge 0$ should be
over one of the following Grassmannians.
The values 
$\delta$ and $\iota$ are also listed.
\def\arraystretch{1.2}
\begin{align*}
\begin{array}{c|rrrrr}
G(k,n) & G(3,8) &G(4,9) &G(4,10) &G(5,11)\\
\midrule
\delta & 5 & 10 & 9 & 15\\
\iota  & 2 & 3  & 0 & 1 \\
\end{array}
\end{align*}
We note that $\iota \le k-1$ holds for all the cases.
Hence
we have
\begin{align}
c_1(\scE') \leq n- c_1(\scE)=\iota \leq k-1
\end{align}
for any irreducible component $\scE' $ of $\scF$ other than $\scE$.
From \cref{lem_remainS},
such $\scE'$ is $\scO(p)$ for some $0<p\le k-1$,
$\wedge^2 \scS^*$ or $\wedge^{k-1}\scS^* $.
Therefore, we obtain 
the remaining ten Calabi--Yau manifolds among ($\alpha$1)--($\alpha$12).
This completes the proof.
\end{proof}
\subsection{The case of \texorpdfstring{$\scF^\scQ= 
  (\wedge^{l-1}\scQ)^{\oplus A}\oplus
  \left(\scQ(1)\right)^{\oplus B}$}{}}
\label{sec:S}
From \cref{prop:Q},
we may focus on the case of
\begin{align}
  \label{Qform}
  \scF^\scQ=(\wedge^{l-1}\scQ)^{\oplus A}\oplus
  \left(\scQ(1)\right)^{\oplus B},
\end{align}
where $A$ and $B$ are non-negative integers.
Note that the assumptions \cref{ass:2} and \cref{ass:5}
imply $n\ge2k+1$ for $(A,B)\ne(0,0)$ and $n\ge 2k$ for $(A,B)=(0,0)$.
We set $\scG:=\scF^\scS \oplus \scF^{\text{line}}$ for simplicity. 
\begin{lem}
  \label{lem:AB}
  Let $\scF = \scF^Q \oplus \scG$
   be a vector bundle as in \cref{mainprop}, which satisfies 
   \cref{conditions_en} and \cref{ass:1}--\cref{ass:5}, and
   assume that $\scF^Q = (\wedge^{l-1}\scQ)^{\oplus A}\oplus
   \left(\scQ(1)\right)^{\oplus B}$. Then the
   possible values of $(A,B)$ are only
$(2,0),(1,0),(0,1)$, and $(0,0)$.
\end{lem}
\begin{proof}
  We note that $\scG$ must satisfy the conditions
\begin{align}
  \label{eq:Gcond1}
  \rank \scG&=kl-(A+B)l-d,\\
  \label{eq:Gcond2}
  c_1(\scG)&=n-A(l-1)-B(l+1),
\end{align}
where $d = \dim Z_\scF$ is a pre-fixed integer.
Since $\scG$ is globally generated,
we have $c_1(\scG) \geq 0$.
If $A\ge 3$, \cref{eq:Gcond2} gives $0\le n-3(l-1)=-2n+3k+3$. 
This contradicts the assumption \cref{ass:1} and $n\ge2k+1$.
Similarly, \cref{eq:Gcond2} gives $0\le -n+2k-2$ (resp.\ $0\le -n+2k$)  if $B\ge 2$ (resp.\ if $A\ge 1$ and $B\ge 1$).
They contradict $n \ge 2k+1$ for $(A,B) \neq (0,0)$.  
\end{proof}

\begin{lem}
  \label{prop:A2B0}	
  If $(A,B)=(2,0)$,
 $\scF$ is one of the following.
\def\arraystretch{1.2}
\begin{align*}
\begin{array}{cccr}
& G(k,n) & \scF & d=\dim Z_\scF \\
\midrule
\text{($\gamma k$)} & G(k,2k+1) &
\left(\wedge^{k}\scQ\right)^{\oplus 2}\oplus \scO(1) & k^2-k-3\\
\text{($\delta k$)} & G(k,2k+2) &
\left(\wedge^{k+1}\scQ\right)^{\oplus 2} & k^2-4\\
\end{array}
\end{align*}

\end{lem}
\begin{proof}
  From \cref{eq:Gcond2} and $l \ge k+1$, 
  we see $c_1(\scG)=k+2-l=0$ or $1$.
  From \cref{lem_remainS},
  one has $\scG= \scF^{\text{line}}$.
  Therefore we also obtain ($\gamma k$) if $c_1(\scG)=1$
  and ($\delta k$) if $c_1(\scG)=0$.
\end{proof}


Now we consider the case of $A+B\le 1$.
Getting the idea from \cite{kuchle},
we consider an invariant 
$\kappa_\scE:=k c_1(\scE)-\rank \scE$
for a vector bundle $\scE$ on $G(k,n)$.
This is an additive integral invariant, 
and takes a positive value
\begin{align}
  \label{eq:ke}
  \kappa_\scE=(\abs \lambda-1)\rank\scE
\end{align}
for any irreducible component $\scE=S_\lambda \scS^* \subset \scG$.
In particular, we have conditions $\kappa_\scE \le \kappa_\scG$ and 
\begin{align}
  \label{eq:Gcond3}
  \kappa_\scG&=k^2+d-A(kl-k-l)-B(kl+k-l),
\end{align}
from \cref{eq:Gcond1} and \cref{eq:Gcond2}.

\vspace{2mm}
In \cref{prop:Q,prop:A2B0},
we describe $\scF$ explicitly for \emph{all} $d >0$.
On the other hand,
we only show the finiteness of $\scF$
for \emph{a fixed} $d >0$ in the following lemmas.
Hence  we fix a positive integer $d >0$
in the rest of this section.

\begin{lem}
  \label{finiteness}
  For a fixed $k$,
  there are at most finitely many 
  $\scF$ with $A+B \leq 1$.
\end{lem}
\begin{proof}
Let $a_\lambda\ge 0$ be the multiplicity of the
irreducible component $S_\lambda \scS^* \subset \scG$
for $\lambda=(\lambda_1\ge \dots \ge \lambda_{k}\ge 0)$.
Note that $\lambda$ can be $(p,p,\dots ,p)$ with $p > 0$,
which corresponds to the line bundle $\scO(p) \subset \scF^{\mathrm{line}}$.
%
%
Since $k$ is fixed and $k-(A+B) \geq 2-1=1$,  $l$ and hence $n$ are uniquely determined from $A, B$ and
$(a_{\lambda})_{\lambda}$ by \cref{eq:Gcond1}.
Hence it suffices to show the finiteness of the choices of $(a_{\lambda})_{\lambda}$.

If $A+B=1$, it holds that
\begin{align}
  \label{cond_c1}
  \sum_{\lambda}a_\lambda c_1(S_\lambda \scS^*)= 
  c_1(\scG)=k\pm 1
\end{align}
from \cref{eq:Gcond2}.
The finiteness of the choices of $(a_\lambda)_{\lambda}$
follows from \cref{cond_c1} and
\begin{align}
c_1(S_\lambda \scS^*) =\frac{\abs \lambda}{k} \rank S_\lambda \scS^* \geq \frac{\abs \lambda}{k} >0.
\end{align}

If $A=B=0$,
\begin{align}
  \label{cond_kappa}
  \sum_{\lambda} a_\lambda \kappa_{S_\lambda \scS^*}
  = \kappa_\scG =k^2+d
\end{align}
holds from \cref{eq:Gcond3}.
The finiteness of the choices of $(a_\lambda)_{\lambda}$
also follows from
\begin{align}
\kappa_{S_\lambda \scS^*} = (\abs \lambda-1)\rank\scE \geq \abs \lambda-1 > 0.
\end{align}
\end{proof}

From \cref{finiteness},
it suffices to show the boundedness of $k$ 
for $A+B\le 1$.
Let us define
\begin{align}
  \label{varphi}
  \varphi_\scE(k):=\kappa_\scE-k^2+A(k^2-k-1)+B(k^2+k-1)
\end{align}
for
each $\scE=S_\lambda \scS^*\subset \scG$.

 From \cref{eq:Gcond3} and the assumption $l \ge k+1$ for $A+B>0$, we have
\begin{align}
 d &= \kappa_\scG - k^2 + A(kl-k-l) + B(kl+k-l)\\
   &\ge \kappa_\scE - k^2 + A(k^2-k-1) + B(k^2+k-1) = \varphi_\scE(k).
\end{align}
Thus we obtain a necessary condition
\begin{align}
  \label{Scond}
	\varphi_\scE(k) \le d
\end{align} 
for any irreducible component $\scE$ in $\scG$ for $A+B=0$ or $1$.

\begin{lem}
  \label{lem:AB1}
  For a sufficiently large $k$,
  there is no $\scF$ with $A+B=1$.
\end{lem}
\begin{proof}
Note that we may assume $\scG\ne 0$. In fact,
if $\scG=0$, \cref{eq:Gcond1} gives the condition $l(k-1)=d$ and hence
$k \le \sqrt{d+1}$.

For $(A,B)=(0,1)$,
we have $0 < \kappa_\scE\le-k+d+1$ for any irreducible
$\scE \subset \scG$ from \cref{Scond}.
Hence $k\le d$ holds.

For $(A,B)=(1,0)$, 
we also have 
\begin{align}\label{eq:k+d+1}
\kappa_{\scE} \le k+d+1
\end{align}
from \cref{Scond}.

Let us assume $\scF^\scS \ne 0$, first.
For $\scE \subset \scF^\scS$,
\begin{align}
  \abs \lambda \le 1+ \frac{k+d+1}{\rank \scE}
\end{align}
holds from \cref{eq:ke}.
If $k > d+1$,
this implies $\abs \lambda \leq 2$
since $\rank\scE \ge k$. 
Since $\scE \neq \scS^*$,
we have $\abs \lambda  =2$ and hence $\scE=\wedge^2 \scS^*$ or $\Sym^2 \scS^*$. 
From \cref{eq:k+d+1} and \cref{eq:ke}, we have
\begin{align}
  2k>k+d+1 \ge \kappa_\scE \ge \rank \scE \ge \binom{k}{2}
\end{align}
if $k > d+1$.
By solving it, we obtain 
$k\le 4$.
Thus we have $k \le \max \{d+1,4\}$ if $\scF^\scS \ne 0$.

In the remaining case
$\scG=\scF^{\text{line}}\ne 0$,
from $\rank\scF^{\text{line}}\le 
c_1(\scF^{\text{line}})$ with \cref{eq:Gcond1} and \cref{eq:Gcond2} one has
 \begin{align}
  0 &\ge \rank \scF^{\text{line}} - c_1(\scF^{\text{line}})\\
  & = (k-1)l -d - (n-l+1)\\
  & \ge (k-1)(k+1) -d - (k+1)\\
  & = k^2 -k - 2 -d.
 \end{align}
Hence it holds $k\le \frac{1}{2}\left( 1+\sqrt{4d+9} \right)$.
\end{proof}


Finally we consider the most intricate case 
of $(A,B)=(0,0)$, i.e.\ $\scF^\scQ=0$.
Contrary to \cref{prop:Q,prop:A2B0,lem:AB1},
we should allow $n=2k$ in this case 
(recall that \cref{ass:5} is the condition that $\scF^\scQ=0$ if $n=2k$).

\begin{lem}
  \label{lem:Smain}
  Assume $\scF^\scQ=0$. 
  For a sufficiently large $k$,
  any irreducible
  component $\scE \subset \scF^\scS$ is one of
 $\wedge^2 \scS^*$, $\Sym^2 \scS^*$, $\wedge^{k-1} \scS^*$, or $ \scS^*(1)$.
\end{lem}
\begin{proof}
	Let $\scE = S_\lambda \scS^*(p) \neq 
	\wedge^2 \scS^*$, $\Sym^2 \scS^*$, $\wedge^{k-1} \scS^*$ nor $ \scS^*(1)$.
	We show that
	each $\varphi_\scE(k)$ is bounded from below by a polynomial
	in $k$ with a positive coefficient in the highest degree,
	and hence, $k$ is bounded by the condition \cref{Scond}.
	We may take $\lambda=(2,1,\dots, 1)$ if $\scE = \scS^*(p')$ for some $p'$,
	and take $\lambda $ with $\lambda_{k}=0$ otherwise.

	If $\lambda \neq 
	(1,1,0,\dots,0)$, $(2,0,\dots,0)$, $(1,\dots,1,0)$, $(2,1,\dots, 1)$,
	$(2,\dots, 2, 0)$, $(1,\dots, 1,0,0)$, \\$(1,1,1,0,\dots, 0)$, $(1,\dots, 1,0,0,0)$,
	then $\rank \scE \geq k^2 -1$ holds
	by \cite[Table 1]{MR0267040}.
	From \cref{eq:ke}, \cref{varphi},  we have
		\begin{equation}
					\varphi_\scE(k)= (\abs{\lambda} +kp -1)\rank \scE - k^2 
	\end{equation}
	for $\scE = S_\lambda \scS^*(p)$.
	Hence for $p \geq 0$, $\varphi_\scE(k) $ is bounded from below by 
	\begin{equation}
					\varphi_{S_\lambda \scS^*} (k)= (\abs{\lambda}-1)\rank \scE - k^2  \ge k^2 -2
	\end{equation}
	since $\abs{\lambda}\ge 3$.
	In the case of
	$\lambda = (2,\dots, 2, 0)$, $(1,\dots, 1,0,0)$, 
	$(1,1,1,0,\dots, 0)$, $(1,\dots, 1,0,0,0)$, 
        $\varphi_\scE(k) $ is bounded from below by  $\varphi_{S_\lambda \scS^*} (k)$, which is equal to
	\begin{align}
					k^3-\frac{3}{2}k^2 -\frac{3}{2}k, \ \ 
					\frac{1}{2}k^3 - 3 k^2 + \frac{3}{2}k, \ \ 
					\frac{1}{3}k^3 - 2 k^2 + \frac{2}{3}k, \ \ 
					\frac{1}{6}k^4 - \frac{7}{6} k^3 + \frac{4}{3}k^2 - \frac{4}{3}k,
	\end{align}
	respectively.
	Finally, for $p \ge 1$ with $\lambda =
	(1,1,0,\dots,0)$, $(2,0,\dots,0)$, $(1,\dots,1,0)$, $(2,1,\dots, 1)$,
	$\varphi_\scE(k) $ is bounded from below by  $\varphi_{S_\lambda \scS^* (1)} (k)$, which is equal to
	\begin{align}
					\frac{1}{2}k^3 -k^2 -\frac{1}{2}k, \ \ 
					\frac{1}{2}k^3 + \frac{1}{2} k, \ \ 
					k^2-2k, \  \ 
					k^2,
	\end{align}
	respectively.
	This completes the proof.
\end{proof}

\begin{proof}[Proof of \cref{mainprop}]
By \cref{lem_F^Q=0},
it suffices to show the finiteness of $\scF$ which satisfies \cref{conditions_en} and 
\cref{ass:1}--\cref{ass:5}.
By \cref{prop:Q,prop:A2B0,finiteness,lem:AB1},
we see the finiteness of such $\scF$ with $\scF^\scQ \neq 0$.

By \cref{finiteness,lem:Smain},
it suffices to show 
that for a sufficiently large $k$, there is no $\scF$ such that
$\scF^\scQ=0$ and $\scF^\scS$ has the form
\begin{align}
  \label{eq:FS}
	\scF^\scS=\left( \wedge^2 \scS^* \right)^{\oplus \alpha}
	\oplus \left( \Sym^2 \scS^* \right)^{\oplus \beta}
	\oplus \left( \wedge^{k-1} \scS^* \right)^{\oplus \gamma}
	\oplus \left( \scS^*(1) \right)^{\oplus \delta},
\end{align}
where $\alpha, \beta, \gamma$ and $\delta$ are non-negative
integers.
In fact,
we show that there is no such $\scF^{\scS}$
if $k > \max\{3+ \sqrt{9+2d},d+2 \}$ as follows.

For such $\scF^{\scS}$,
we have 
  \begin{align}
  \label{eq:FSrank}
	\rank \scF^\scS&=\alpha \binom{k}{2}+\beta
	\binom{k+1}{2}+\gamma k +\delta k,\\
  \label{eq:FSc1}
	c_1(\scF^\scS)&=\alpha(k-1)+\beta(k+1)+\gamma(k-1)+\delta(k+1),\\
  \label{eq:FSkappa}
	\kappa_{\scF^\scS}&=\alpha \binom{k}{2}+\beta
	\binom{k+1}{2}+\gamma k(k-2) +\delta k^2.
  \end{align}
  By using the condition $\kappa_{\scF^\scS} \le \kappa_\scF=k^2+d$,
  we obtain
  \begin{equation}
\begin{split}
(\alpha&+\beta+2\gamma+2\delta-2)k^2+(-\alpha+\beta-4\gamma)k\\
&=	(\alpha+\beta+2\gamma+2\delta-2)k(k-2)+(\alpha+3\beta+4\delta-4)k
  \le 2d. \label{eq:k^2+d}
  \end{split}
    \end{equation}
 If $\alpha+\beta+2\gamma+2\delta > 2$,
 we have
 $k(k-2) -4k \leq 2d$, i.e.\ $k \leq 3+ \sqrt{9+2d}$.
Thus 
it holds that $\alpha+\beta+2\gamma+2\delta\le 2$ if $ k > 3+ \sqrt{9+2d}$,
  and hence
$(\alpha,\beta,\gamma,\delta)$ is one of the following;
  \begin{align}\label{eq_alpha-delta}
  (2,0,0,0), (1,1,0,0),(0,2,0,0),(0,0,1,0),(0,0,0,1),(1,0,0,0),
  (0,1,0,0), (0,0,0,0).
  \end{align} 
  
Since $\scF=\scF^\scS \oplus \scF^{\text{line}}$, it holds that
  \begin{align}\label{eq_rank_c_1}
	\rank \scF^{\text{line}}&=kl-\rank \scF^\scS-d,\\
	c_1(\scF^{\text{line}})&=n-c_1(\scF^\scS).
  \end{align}
Furthermore, 
$\rank \scF^{\text{line}} \le c_1(\scF^{\text{line}})$
must hold.
Hence
\begin{align}
k^2 -2k -\rank \scF^\scS +c_1(\scF^\scS) \leq d
\end{align}
holds since $k^2 -2 k \leq kl-n $ by \cref{ass:2}.
Thus we have 
  \begin{align}\label{eq_2d}
	(2-\alpha-\beta)k^2+(-4+3\alpha+\beta)k
  +2(-\alpha+\beta-\gamma+\delta) \le 2d.
  \end{align}
  If $k > d+2$, we see that
  \cref{eq_2d} is satisfied only for 
  $(\alpha, \beta,\gamma,\delta)=(1,1,0,0)$ or
  $(0,2,0,0)$ among \cref{eq_alpha-delta}. 
  
Finally, we use the invariant $\kappa$.
It holds that
\begin{align}\label{eq_kappa_line}
\kappa_{\scF^{\text{line}}}=\kappa_{\scF} -\kappa_{\scF^\scS} = k^2 +d  -\kappa_{\scF^\scS}  = \left\{ \begin{array}{ll}
    d & \text{if } (\alpha, \beta,\gamma,\delta)=(1,1,0,0) \\
    d-k & \text{if } (\alpha, \beta,\gamma,\delta)=(0,2,0,0).
  \end{array} \right. 
\end{align}
Since $\kappa_{\scO(p)}=pk-1 \geq k-1$ for $p > 0$,
we have $\kappa_{\scF^{\text{line}}}\geq k-1 $ if $\scF^{\text{line}}\neq 0$ and  $\kappa_{\scF^{\text{line}}} =0$ if $\scF^{\text{line}} =0$.
Hence \cref{eq_kappa_line} leads to a contradiction for $k > d+1$. 

  Therefore there exists no choice of $\scF$ for 
  a sufficiently large $k$. This completes the proof.
\end{proof}



\section{Complete intersection Calabi--Yau \texorpdfstring{$3$}{}-folds}
\label{sec:CY3}
Let us apply our method to the classification of 
complete intersection Calabi--Yau $3$-folds and show \cref{mainthm_en}.

\begin{rmk}
\label{ss:natural}
In \cref{tb:cy3},
we omit some duplication originated from 
several kinds of natural identifications. 

There are the natural isomorphisms between $G(k,n)$ and $G(n-k,n)$,
$Z_{\scS^*}\subset G(k,n)$ and $G(k,n-1)$, $Z_{\scQ}\subset G(k,n)$
and $G(k-1, n-1)$, and
$Z_{\wedge^2 \scQ}
\subset G(k,2k)$ and $Z_{\wedge^2 \scQ} \subset G(k, 2k+1)$
discussed in \cref{sec:finite}.

We have one more identification,
which was not considered in \cref{sec:finite}.
On $LG(k,2k)\simeq Z_{\wedge^2 \scS^*} \subset G(k,2k)$,
$\scS^*|_{Z_{\wedge^2 \scS^*}}$ and $\scQ|_{Z_{\wedge^2 \scS^*}}$ are isomorphic.
Hence we identify
$ Z_{\wedge^2 \scS^* \oplus S_{\lambda} \scQ \oplus \scF'}$ with
$Z_{\wedge^2 \scS^* \oplus S_{\lambda} \scS^* \oplus \scF'}
\subset G(k,2k)$.
\end{rmk}

Let $\scF=\scF^\scS \oplus \scF^\scQ \oplus \scF^{\mathrm{line}}$ 
be a homogeneous vector bundle on a 
Grassmannian $G(k,n)$ that satisfies
\begin{align}
  \label{conds}
  \rank \scF =k(n-k)-3 \quad \text{and} \quad c_1(\scF)=n
\end{align}
and the assumptions \cref{ass:1}--\cref{ass:5}.

First we consider the case $\scF^\scQ \neq 0$.
If $\scF^\scQ\ne (\wedge^{l-1}\scQ)^{\oplus A}\oplus \left(\scQ(1)\right)^{\oplus B}$,
we obtain two Calabi--Yau $3$-folds,
($\alpha$2) and ($\beta$2) from \cref{prop:Q},
i.e.\ \cref{row29} and \ref{row5} in \cref{tb:cy3}, respectively.

Next, we consider the case of $\scF^\scQ=
(\wedge^{l-1}\scQ)^{\oplus A}\oplus \left(\scQ(1)\right)^{\oplus B}$.
From \cref{lem:AB}, the possible values are $(A,B)=(2,0),(1,0),(0,1)$ or
$(0,0)$. For $(A,B)=(2,0)$, \cref{prop:A2B0} gives
a Calabi--Yau $3$-fold ($\gamma$3) i.e.\ \cref{row24}.
For $A+B=1$, there is a bound $k \le 4$
as in the proof of \cref{lem:AB1}.
For each $k\le 4$, we use the other condition
\cref{cond_c1} and obtain seven Calabi--Yau $3$-folds,
\cref{row10,row11,row15,row16,row17,row18,row25}.

Let $\scF^\scQ=0$.
Suppose that $k\le 5$, first.
From 
\begin{align}
(\abs \lambda -1)k \le (\abs \lambda -1) \rank \scE =\kappa_\scE \le \kappa_\scF =k^2+3
\end{align}
for an irreducible bundle $\scE=S_\lambda \scS^*
\subset \scF^\scS$, we only need to check the Young diagrams
with $\abs \lambda \le 4$ for $k=2$, $\abs \lambda \le 5$ 
for $k=3, 4$ and $\abs \lambda \le 6$ for $k=5$. We restrict 
the possible bundles further by the actual evaluation of
$\kappa_\scE \le k^2+3$.
For instance, 
if $k=2$, $\scF$ must be of the form
\begin{align}
  \scF=\scS^*(1)^{\oplus x_1}\oplus \left(\Sym^2 \scS^*
  \right)^{\oplus x_2} \oplus \scO(1)^{\oplus y_1}
  \oplus \scO(2)^{\oplus y_2}\oplus
  \scO(3)^{\oplus y_3} \oplus 
  \scO(4)^{\oplus y_4}.
\end{align}
The conditions \cref{conds} turn into
\begin{align}
  7&=\kappa_\scF=4x_1+3x_2+y_1+3y_2+5y_3+7y_4,\\
  4&\le n=3x_1+3x_2+y_1+2y_2+3y_3+4y_4.
\end{align}
By solving them, we get the remaining Calabi--Yau $3$-folds
among \cref{row1}--\cref{row18}.
The same argument works for each $k\le 5$ and
we obtain the remaining cases
among \cref{row19}--\cref{row33}.

Assume $\scF^{\scQ}=0$ and $k \geq 6$.
In this case,
$\scF^\scS$
must be of the form 
$\left( \wedge^2 \scS^* \right)^{\oplus \alpha}
\oplus \left( \Sym^2 \scS^* \right)^{\oplus \beta}
\oplus \left( \wedge^{k-1} \scS^* \right)^{\oplus \gamma}
\oplus \left( \scS^*(1) \right)^{\oplus \delta}$
for some $\alpha,\beta,\gamma$ and $\delta$,
by checking each bound of $\varphi_\scE(k)$ by the polynomial
in the proof of \cref{lem:Smain}.
By the proof of \cref{mainprop}, we have
\begin{align}
k \leq \max \{ 3+\sqrt{9+2d}  , d+2\} =\max \{  3+\sqrt{15} , 5 \} < 7,
\end{align}
and hence $k=6$.
We note that \cref{eq:k^2+d}, \cref{eq_2d} and \cref{eq_kappa_line} hold 
even if $k \leq \max \{ 3+\sqrt{9+2d}, d+2\} $.
By inequalities \cref{eq:k^2+d}, \cref{eq_2d} for $d=3,k=6$,
we see that $(\alpha,\beta,\gamma,\delta)$ must be $(1,1,0,0)$.
Hence $\kappa_{\scF^{\text{line}}} = d=3$ by \cref{eq_kappa_line}.
However there is no such $ \scF^{\text{line}}$ since $ \kappa_{\scO(p)}=pk-1 \ge 5 > 3$ for $p >0$.
Thus there is no solution for $k \geq 6$.


\vspace{2mm}
Finally,
we consider $\scF$ 
that does not satisfy \cref{ass:5},
i.e.\ $\scF^\scQ \ne 0$ with $n=2k$. Let $\scF$ be a classified
homogeneous vector bundle over $G(k,2k)$
which satisfies \cref{conds} and $\scF^\scQ=0$. 
As in \cref{eg_A5},
we can obtain another $\scF'$ on $G(k,2k)$ by
replacing one of $S_{\lambda} \scS^* \subset \scF$ with $S_{\lambda} \scQ$.
If $\scF^\scS$ is irreducible,
such $\scF'$ dose not give a new family,
since $Z_{S_{\lambda} \scS^* \oplus  \scF^{\mathrm{line}}}$ and $Z_{S_{\lambda} \scQ \oplus  \scF^{\mathrm{line}}}$ in $G(k,2k)$ are identified.
Otherwise, i.e.\ for \cref{row21,row31,row33},
$\scF^{\scS}$ contains $\wedge^2 \scS^*$ as an irreducible component.
Hence we do not obtain a new family in these cases as well
because of the last identification in \cref{ss:natural}.
Thus we do not have $\scF$ 
with $\scF^\scQ \ne 0$ and $n=2k$ in \cref{tb:cy3}.

\vspace{2mm}
The Hodge numbers of the Calabi--Yau $3$-folds can be calculated by 
the similar way as in
\cite[Section\,4.4]{kuchle}.
All examples except for \cref{row30} have $h^{0,0}(Z_{\scF})=1$, 
i.e.\ they are 
irreducible $3$-folds.
For \cref{row30}, we have $h^{0,0}(Z_{\scF})=2$, which is consistent
with the description in \cref{sec_description4}.
We also have $h^1(\scO_{Z_{\scF}})=0$ for all examples
except for \cref{row26}. 
This means that they are Calabi--Yau $3$-folds in 
the strict sense, while \cref{row26} is an abelian 3-fold
as we see in \cref{sec_description4}.




\section{Alternative description: 
\texorpdfstring{\cref{row4,row5,row7,row10}}{}}\label{sec_description}

In the rest of this paper,
we study alternative descriptions of some of the Calabi--Yau $3$-folds in Table \ref{tb:cy3}.
In this section,
we treat \cref{row4,row5,row7,row10}.\\

Let $X\subset  \P(V)$ be a projective variety and 
let $\cale $ be a globally generated locally free sheaf on $X$.
We consider a description of $Z_{\scE(1)} \subset X$.

Let $E \subset H^0(\cale)$ be a subspace which globally generates $\cale$.
We denote by $\calk$ the kernel of the natural surjection $ E\otimes \calo_X \arw \cale$.
We consider the following diagram:
\begin{equation}\label{diagram_pi_mu}
\begin{gathered}
\xymatrix{
	&  \P:=\P_{X} \left(\calk \oplus \calo_X(-1 ) \right) \ar[ld]_{\mu} \ar[rd]^{\pi} &   &\\
	\P(E \oplus V ) \ar@/_15pt/@{-->}[rrr]&	 & X \ar@{^(->}[r] & \P(V),\\
}
\end{gathered}
\vspace{4mm}
\end{equation}
where $\pi$ is the projective bundle induced by $\calk \oplus \calo(-1)$,
$\mu$ is the morphism induced by $\calk \oplus \calo(-1) \subset (E \oplus V) \otimes \calo_{X}$,
and $\P(E \oplus V) \dashrightarrow \P(V)$ is the natural projection.
Set 
\begin{align}\label{eq_def_sigma}
\Sigma =\mu(\P) \ \subset \ \P(E \oplus V ).
\end{align}
We note that
\begin{align}\label{eq_tot}
\P \setminus \P_X(\calk ) =\mathrm{Tot}_X (\calk \otimes \calo(1)),  \quad  \P(E \oplus V) \setminus \P(E) = \mathrm{Tot}_{\P(V)}(E \otimes \calo(1)),
\end{align}
where $\mathrm{Tot}$ stands for the total space of a vector bundle,
and $\mu|_{\P \setminus \P_X(\calk )}$ is induced by the embedding $\calk \otimes \calo(1) \hookrightarrow E \otimes \calo(1)$.
Hence $\mu$ is an embedding outside $ \P_X(\calk )$ and induces an isomorphism onto the image 
$\P \setminus \P_X(\calk ) =\mathrm{Tot}_X (\calk \otimes \calo(1)) \stackrel{\sim}{\longrightarrow}  \Sigma \setminus (\Sigma \cap \P(E)) $.
By this isomorphism,
we have a description
\begin{align}\label{eq_explicit_description}
 \Sigma \setminus (\Sigma \cap \P(E)) 
 =\left\{  [(u,v)] \,\big| \,  [v] \in X,  u([v]) =0 \in \cale_{[v]}  \right\}  ,
\end{align}
where $ u \in   E, v \in V \setminus 0$, and $[v] \in \P(V)$ is the corresponding point.

\vspace{2mm}
Let $\bar{s}  $ be an element in $E \otimes V^*$.
Let $s \in H^0(\cale(1))$ be the image of $\bar{s}$ by the natural map $E \otimes V^* \arw H^0(\cale(1))$.

Since $\bar{s} \in  E \otimes V^*$,
we have the corresponding linear map $V \arw  E$,
which we also denote by the same letter $\bar{s} $.
Hence we have a linear embedding
\begin{align}\label{qe:embedding}
\P(V) \hookrightarrow \P(E \oplus V)  \quad  :  \quad [v] \mapsto \big[(\bar{s}(v), v) \big] .
\end{align}
Let $P_{\bar{s}} \subset \P(E \oplus V )$ be the image of the embedding.
We note that $ P_{\bar{s}}  \cap \P(E) = \emptyset$.

Conversely,
if a linear subvariety $P \subset \P(E \oplus V )$ of dimension $\dim \P(V)$ satisfies $P \cap \P(E) =\emptyset$,
there exists  $\bar{s} \in E \otimes V^*$ such that $P=P_{\bar{s}} $.
Hence $P_{\bar{s}} \subset \P(E \oplus V)$ is general if so is $\bar{s}$.

\begin{prop}\label{prop_Z=Z}
Let $\bar{s}  $ be an element in $E \otimes V^*$,
and let $s$, $P_{\bar{s}}$ be as above.
Let $Z \subset X$ be the zero locus of the section $s \in H^0(\cale(1))$ 
and let $ Z' \subset \Sigma$ be the linear section of $\Sigma$ by $P_{\bar{s}}$.
Then $Z \subset \P(V)$ and $Z' \subset P_{\bar{s}}$ are projectively equivalent.
\end{prop}

\begin{proof}
Consider the diagram
\begin{equation}\label{diagram_tilde}
\begin{gathered}
\xymatrix{
	&  \widetilde{\P}:=\P_{X} \left(E \otimes \calo_X  \oplus \calo_X(-1 ) \right) \ar[ld]_{\tilde{\mu}} \ar[rd]^{\tilde{\pi}} &  \\
	\P(E \oplus V )&	 & X ,\\
}
\end{gathered}
\end{equation}
and two vector bundles $ \calf_1:=\tilde{\pi}^* \cale \otimes \tilde{\mu}^* \calo(1), \calf_2:= E \otimes \tilde{\mu}^*\calo(1)$ on $ \widetilde{\P}$.
Let $c: E \otimes \calo_X \rightarrow \cale$ be the canonical morphism.
Then the composition of $\tilde{\mu}^*\calo(-1) \hookrightarrow \tilde{\pi}^* (E \otimes \calo_X  \oplus \calo_X(-1 ) ) $ and the pullback of $(c,0) : E \otimes \calo_X  \oplus \calo_X(-1 ) \rightarrow \cale$ by $\tilde{\pi}^*$
defines a global section $s_1 \in H^0(\calf_1)$.
Since $ \ker (c,0) = \calk \oplus \calo_X(-1 )$,
the zero locus $Z(s_1) \subset  \widetilde{\P}$ is nothing but $\P=\P_{X} \left(\calk \oplus \calo_X(-1 ) \right)  $.

On the other hand,
an element $\bar{s}  \in E \otimes V^*$ induces an exact sequence
\begin{align}\label{eq_s2}
0 \rightarrow \calo_X(-1) \stackrel{(\bar{s},\id)}{\longrightarrow} E \otimes \calo_X  \oplus \calo_X(-1 ) \stackrel{(-\id,\bar{s})}{\longrightarrow} E  \otimes \calo_X \rightarrow 0
\end{align}
on $X$.
Let $s_2 \in H^0(\calf_2) =\Hom(\tilde{\mu}^*\calo(-1), E  \otimes  \calo_{\widetilde{\P}})$ be the global section corresponding to
the composition of $\tilde{\mu}^*\calo(-1) \hookrightarrow \tilde{\pi}^* (E \otimes \calo_X  \oplus \calo_X(-1 ) ) $ and $(-\id,\bar{s}) $.
Then the zero locus $Z(s_2)$ coincides with $ \P_X(\calo_X(-1) )$, which is embedded in $\widetilde{\P} $
by $ (\bar{s},\id) : \calo_X(-1) \rightarrow E \otimes \calo_X  \oplus \calo_X(-1 )$.
Under the identification of $ Z(s_2)= \P_X(\calo_X(-1) )$ with $X$ by $\tilde{\pi} |_{Z(s_2)} $,
we have $\calf_1 |_{Z(s_2)} = \cale(1)$ and $s_1 |_{ Z(s_2)} = s \in H^0(\cale(1))$.
Hence 
\begin{align}\label{eq_pi11}
\tilde{\pi} |_{Z(s_1|_{Z(s_2)})} : Z(s_1 |_{ Z(s_2)} ) \rightarrow Z
\end{align}
is an isomorphism.

By \cref{qe:embedding}, \cref{eq_s2},
we also have $Z(s_2) =\tilde{\mu}^{-1} (P_{\bar{s}})$.
Thus $\tilde{\mu} $ induces
\begin{align}\label{eq_mu12}
\tilde{\mu}|_{Z(s_1) \cap Z(s_2) } :  Z(s_1) \cap Z(s_2)  = \P \cap Z(s_2)  \rightarrow \Sigma \cap P_{\bar{s}}=Z',
\end{align}
which is an isomorphism 
since $\mu=\tilde{\mu} |_{\widetilde{\P}} : \widetilde{\P} \rightarrow \Sigma $ is an isomorphism outside  $ \P(E )  $ and $ \P(E )   \cap P_{\bar{s}} = \emptyset$.
Since $ Z(s_1 |_{ Z(s_2)} ) = Z(s_1) \cap Z(s_2)  $,
two isomorphisms \cref{eq_pi11}, \cref{eq_mu12} induce an isomorphism $Z \rightarrow Z'$.
By construction,
this isomorphism is the restriction of $\P(V) \stackrel{\sim}{\rightarrow} P_{\bar{s}}  \subset \P(E \oplus V)   $ in \cref{qe:embedding}.
Thus $Z \subset \P(V)$ and $Z' \subset P_{\bar{s}}$ are projectively equivalent.
\end{proof}


In the following subsections,
we consider the case when $X \subset \P(V)$ is a Grassmannian with the Pl\"{u}cker embedding and $E=H^0(\cale )$.
In that case,
$H^0(\cale ) \otimes V^* \rightarrow H^0(\cale(1))$ is surjective for a globally generated irreducible homogeneous $\cale$ on $X$.
Hence the section $s \in H^0(\cale(1))$ in \cref{prop_Z=Z} is general if so is $\bar{s}$.

\subsection{\texorpdfstring{\cref{row4,row7}}{}}

In this subsection,
we consider the case $\cale=\scS^*$ on $G(2,W)$ for $5 \le \dim W \le 8$.
In this case,
$E=H^0(\cale)=W^*, \calk =\scQ^*$,
and  $\Sigma$ 
is contained in $ \P(W^* \oplus \wedge^2 W) $.
 We show below that in this case $\Sigma$ (defined in \cref{eq_def_sigma}) can be identified 
 with a certain Schubert variety in a homogeneous space of a simple Lie group $G$
 of type $E_n$, $5 \le n \le 8$ (with the standard convention $E_5 = D_5$).
 We fix a Borel subgroup $B\subset G$.

In Dynkin diagrams of type $D$ and $E$,				
we use the numeration $1,2,\dots, n-1$ for the nodes 
in the upper row and $n$ for the unique lower node
and denote by $P_i$
the maximal parabolic subgroup corresponding to the node $i$.

For instance, 
in the following Dynkin diagram,
the crossed node indicates the corresponding maximal parabolic subgroup $P_{n-1}\subset E_n$:
\begin{align}
\Fdyn
\end{align}

\begin{lem}\label{prop_Tits}
Suppose that $\dim W =5,6,7$ or $8$.
The variety $\Sigma \subset \P(W^* \oplus \wedge^2 W)$ is 
projectively equivalent to a 
Schubert variety of a generalized Grassmannian $G/P=D_5/P_4, E_6/P_5, E_7/P_6$ or  $E_8/P_7
\subset \P \left(H^0(\scO_{G/P}(1))\right)$, respectively.

\end{lem}

\begin{proof}
We recall the Tits transform.
Consider a diagram
\begin{equation}\label{diagram_Tits}
\begin{gathered}
\xymatrix{
 & G/P_{i,j}  \ar[ld]_{\tilde{\mu}} 
	\ar[rd]^{\tilde{\pi}} &  \\
        G/ P_j &	 & G/ P_i\\
}
\end{gathered}
\end{equation}
where $i \ne j$ and $P_{i,j} = P_i \cap P_j$.
For a subset $X \subset G/P_i$,
the \emph{Tits transform} of $X$ is defined by $\calt (X):=\tilde{\mu} 
(\tilde{\pi}^{-1}(X)) \subset G/P_j$,
(see \cite{LMbook} for detail).
In the following, we show $\scT(X)\simeq \Sigma \subset G/P_{n-1}$
if we take $i=2, j=n-1$ and $X = G(2,n) \subset G/P_2$.

Let $\Delta= \left\{\alpha_1, \dots, \alpha_n\right\}$ be the set of all simple roots of $G$,
where $\alpha_k$ corresponds to the $k$-th node of the Dynkin diagram.
For the above $i$, 
the subset $\Delta' = \Delta \setminus \left\{ \alpha_i \right\} \subset \Delta$ defines
a semi-simple Lie subgroup $G'\subset G$. 
From the discussion of \cite[Section 2.7.1]{LMbook},
the homogeneously embedded homogeneous submanifold
$G'/(P_j \cap G')\subset G/P_j$
is a smooth Schubert variety of $G/P_j$, which coincides with the
Tits transform $\scT (o)$ of the Borel fixed point $o:=P_i/P_i \in G/P_i$.

Applying this to $i=n$, $j=2$,
we obtain the right side of the following diagram \cref{diagram_E6}.
The two crossed nodes
correspond to
the two maximal parabolic subgroups $P_i = P_n$ and $P_j = P_2$, and
the encircled nodes
correspond to the Lie subgroup $G' =A_{n-1}$.
The crossed Dynkin diagrams
and the encircled crossed Dynkin diagram represent
the corresponding 
flag manifolds and 
homogeneously embedded homogeneous submanifold $G(2,n) 
= \scT (o)$, respectively.

The left side of the diagram  \cref{diagram_E6} 
represents another Tits transform of the submanifold $G(2,n)$ for $i=2,j=n-1$. 
\begin{equation}\label{diagram_E6}
\begin{gathered}
\xymatrix{
  & 
  \Fsix 
  \ar[ld]_{\tilde\mu} \ar[rd]^{\tilde\pi} 
  &	& 
  \Ffive 
	\ar[ld]_{\tilde\mu'} \ar[rd]^{\tilde\pi'} &  \\
  \Ffour 
  & \tilde{\pi}^{-1}(G(2,n))
\ar[ld] \ar@{}[u]|{\bigcup} \ar[rd]
& 
\Fone 
& 
G(2,n)\ar@{=}[ld] \ar@{}[u]|{\bigcup} \ar[rd]
& \Fthree 
\\
\scT\left( G(2,n) \right)
\ar@{}[u]|{\bigcup} & &
\Ftwo\ar@{}[u]|{\bigcup} & & o 
\ar@{}[u]|{\rotatebox{90}{$\in$}}\\
}
\end{gathered}
\end{equation}

Let $\varpi_i$ be the $i$-th fundamental weight of $G$ and $V^P_{\varpi_i}$
the irreducible representation with the highest weight $\varpi_i$ 
of a parabolic subgroup $P$. 
Note that an integral dominant weight for $G$ 
also defines a highest weight representation of a parabolic subgroup $P$
induced from that of the Levi group $L$ via the canonical morphism 
$P\rightarrow L$.
Let $\calg$ and $\calg'$ be the homogeneous vector bundles 
corresponding to the fundamental representations
$V_{\varpi_{n-1}}^{P_2}$ and
$V_{\varpi_{n-1}}^{P_{2,n}}$
on $G/P_2$ and $G/P_{2,n}$,
respectively, i.e.\ $H^0(\calg^*)\simeq H^0(\calg'^*)$ are isomorphic to the
fundamental representation $V_{\varpi_{n-1}}$ of $G$.
We have an injection $\calg'\subset \tilde\mu'^* \calg$ on $G/P_{2,n}$ induced by
the injection $V^{P_{2,n}}_{\varpi_{n-1}} \subset V^{P_{2}}_{\varpi_{n-1}}$
as representations of $P_{2,n}$.
Note that the sets of weights of $V_{\varpi_{n-1}}^{P_{2,n}} 
\subset V_{\varpi_{n-1}}^{P_2}$ can be written down as follows,
by subtracting roots of $P_{2,n}$ or $P_2$ from the highest weight $\varpi_{n-1}$,
\begin{align}
\begin{split}
&\left\{\varpi_{n-1}, \varpi_{n-2}-\varpi_{n-1}, \dots, \varpi_3 -\varpi_4,
\varpi_2+\varpi_n - \varpi_3 \right\}\\ & \hspace{10mm} \subset
\left\{\varpi_{n-1}, \varpi_{n-2}-\varpi_{n-1}, \dots, \varpi_3 -\varpi_4,
\varpi_2+\varpi_n - \varpi_3 , \varpi_2 - \varpi_n\right\}.
\end{split}
\end{align}
By computing the determinants of these vector bundles, 
i.e.\ the sum of all weights of these representations, 
and restricting them to  $\tilde\pi'^{-1}(o)=G(2,n)$,
we obtain an exact sequence on $G(2,n)$, 
\begin{align}
		0 \rightarrow \calg'|_{G(2,n)} \rightarrow \tilde\mu'^* \calg|_{G(2,n)} \rightarrow 
		\scO(-1) \rightarrow 0.
\end{align}

Since 
$V_{\varpi_{n-1}}^{P_{2,n}}$ 
coincides with
the irreducible representation $V_{\varpi_{n-1}}^{P_2'}$ 
of the parabolic subgroup 
$P_2':= P_2 \cap G'$ of $G'=A_{n-1}$
as $P_2'$-modules, 
we identify the vector bundle $\calg'|_{G(2,n)}$ with $\scQ^*$.
We also have $\Ext^1(\scO(-1),\scQ^*) = H^1(\scQ^*(1)) =0$
by the Bott--Borel--Weil theorem.
Thus the restriction $\calg |_{G(2,n)}$ is isomorphic to 
the direct sum $\scQ^* \oplus \calo(-1)$ on $G(2,n)\subset G/P_2$.

In the diagram \cref{diagram_E6},
$\tilde{\pi} : G/P_{2,n-1}
\rightarrow G/P_2$ 
is nothing but the
$\P^{n-2}$-bundle $\P_{G/P_2}(\calg) \rightarrow G/P_2$,
and $\tilde{\mu}$ is the morphism induced by 
the tautological line bundle of $\P_{G/P_2}(\calg)$.
Hence we have $\tilde{\pi}^{-1}(G(2,n)) =\P_{G(2,n)}\left(\calg |_{G(2,n)}\right)
= \P_{G(2,n)}\left( \scQ^* \oplus
\calo(-1) \right)$, and 
the diagram
\begin{equation}\label{diagram_tits_sub}
\begin{gathered}
\xymatrix{
  & \tilde{\pi}^{-1}(G(2,n) ) = \P_{G(2,n)}\left( \scQ^* \oplus
\calo(-1) \right) 
\ar[ld]  \ar[rd]
& 
\\
\scT\left( G(2,n) \right)
 & &
\Ftwo =G(2,n)
\\
}
\end{gathered}
\end{equation}
in \cref{diagram_E6} coincides 
with \cref{diagram_pi_mu} for $X=G(2,W),\dim W=n, 
\cale=\scS^*,E=W^*$ and $\calk=\scQ^*$. 
The left morphisms in both \cref{diagram_pi_mu} and \cref{diagram_tits_sub}
are the morphisms induced by the tautological line bundles of the right ones respectively.
Thus it turns out that $\Sigma$ is
the Tits transform
$\calt(G(2,n))$ of a Schubert variety $G(2,n) \subset G/P_2$ in the
diagram \cref{diagram_E6}.
By \cite[Lemma 2.4]{coskun},
a Tits transform of a Schubert variety is also a Schubert variety in general.
Thus $\Sigma$ is a Schubert variety of $G/P_{n-1}$.
\end{proof}

\begin{rmk}\label{rmk_galkin}
By \cite[Lemma 2.4]{coskun},
we can compute the Tits transform $\calt(G(2,n)) \subset G/P_{n-1}$
explicitly.
For each $5\le n\le 8$, 
the list of reduced expressions of the corresponding elements
in the Weyl groups $\bfW_G$ is given in the following table, where 
we abbreviate a reduced expression $w=s_{i_1}\cdots s_{i_m} \in \bfW_G$ to $i_1 \cdots i_m$
for simple reflections $s_1, \dots, s_n  \in \bfW_G$.
We denote by $\bfW_{\text{min}}^{P_j}\subset \bfW_G$  the set of minimal length representatives of the cosets  $\bfW_G/\bfW_{P_j}$  for $j=2$ or $n-1$.

\def\arraystretch{1.3}
\begin{table}[H]
\centering
\begin{align*}
\begin{array}{r|c|c}
				&w\in \bfW_{\text{min}}^{P_{n-1}} \text{ with }
				\Sigma=\overline{Bw^{-1}P_{n-1}/P_{n-1}}&
  w\in \bfW_{\text{min}}^{P_2} \text{ with } G(2,n)=\overline{Bw^{-1}P_2/P_2}\\
  \hline
  G=D_5 & 432153243 & 213243 \\
  E_6 & 543216324354 & 21324354 \\
  E_7 & 654321732435465 & 2132435465 \\
  E_8 & 765432183243546576& 213243546576\\
\end{array}
\end{align*}
\end{table}
From \cref{prop_Z=Z}, \cref{prop_Tits} and \cref{rmk_galkin}, 
we get the descriptions of \cref{row4,row7} in \cref{tb:cy3} as follows.

\begin{prop}
The zero locus of a general section of $\scS^*(1)\oplus \scO(2)$ on $G(2,5)$ is 
isomorphic to the complete intersection of general six hyperplanes and one quadric in $OG(5,10)$.
\end{prop}

\begin{proof}
 Let $n=5$. By \cref{prop_Tits} and \cref{rmk_galkin}, $\Sigma$ is the Schubert divisor of $OG(5,10)$,
\[
\{  U \in OG(5,10) \, | \, U \cap U_0 \neq \emptyset \},
\]
where $U_0$ is a fixed
subspace
corresponding to the other irreducible component of the Grassmannian
of $5$-dimensional isotropic subspaces,
as shown by \cite[Section 16]{CCGK} in a different way.
Especially, 
$\Sigma$ is a special hyperplane section of $OG(5,10)$.
By \cref{prop_Z=Z},
a Calabi--Yau $3$-fold $Z_{\scS^* (1) \oplus \calo(2)  } \subset G(2,5)$ in \cref{row4} is 
isomorphic to $\Sigma_{1^5,2}$,
that is, the complete intersection of $\Sigma $
by general five hyperplanes and a general quadric.

We note that the dual variety $OG(5,10)^* \subset {\P^{15}}^{\vee}$
is isomorphic to $OG(5,10)$ itself,
and each point in $OG(5,10)^*$ corresponds to a Schubert divisor of $OG(5,10)$.
In particular, 
codimension of $OG(5,10)^* $ in ${\P^{15}}^{\vee}$ is five.

By definition, $OG(5,10)_{1^6} = OG(5,10) \cap H_0 \cap \dots\cap H_5 $ for general hyperplanes $H_i \subset \P^{15}$.
Then the linear space $\P^5=\langle H_0 \dots, H_5 \rangle \subset {\P^{15}}^{\vee}$ spanned by the six points $H_0 \dots, H_5  \in  {\P^{15}}^{\vee}$ intersects $OG(5,10)^* $.
Take $H'_0 \in \langle H_0 \dots, H_5 \rangle \cap OG(5,10)^*$ and general $H'_1,\dots, H'_5 \in \langle H_0 \dots, H_5 \rangle$
so that $\langle H_0 \dots, H_5 \rangle = \langle H'_0 \dots, H'_5 \rangle$.
Then we have 
\begin{align*}
OG(5,10)_{1^6} &= OG(5,10) \cap H_0 \cap \dots\cap H_5 \\
&= OG(5,10) \cap H'_0 \cap \dots\cap H'_5  = (OG(5,10) \cap H'_0) \cap H'_1 \cap\dots\cap H'_5.
\end{align*}
Since $OG(5,10) \cap H'_0$ is a Schubert divisor by $H'_0 \in  OG(5,10)^*$,
we see that $OG(5,10)_{1^6}$ is a linear section of a Schubert divisor by general five hyperplanes,
i.e.\ $OG(5,10)_{1^6} = \Sigma_{1^5}$.
Hence 
we have $OG(5,10)_{1^6,2} = \Sigma_{1^5,2}$.
\end{proof}

\begin{prop}

	The zero locus of a general section of $\scS^*(1)\oplus\scO(1)^{\oplus 3}$
	on $G(2,6)$ is isomorphic to a complete intersection of general nine hyperplanes
	in the $12$-dimensional Schubert variety $\Sigma \subset E_6/P_5=\mathbb{O}\P^2$ of degree 33.
\end{prop}

\begin{proof}
In the case of $n=6$, $\Sigma$ is
the $12$-dimensional Schubert variety
of degree $33$ in the Cayley plane $\mathbb{O}\P^2$ by \cref{rmk_galkin},
which is first pointed out by \cite{galkin, galkin2}.
By \cref{prop_Z=Z},
$ Z_{\scS^*(1)} \subset G(2,6)$ is projectively equivalent to $\Sigma_{1^6} $.
Hence $ Z_{\scS^*(1)\oplus\scO(1)^{\oplus 3}} \subset G(2,6) $ is projectively equivalent to $\Sigma_{1^9} $.
\end{proof}

The geometry of the Schubert variety $\Sigma \subset \mathbb{O}\P^2$ 
and the 
Calabi--Yau $3$-fold $\Sigma_{1^9} $
are discussed in detail by \cite{MR3688804}.

\end{rmk}


\subsection{\texorpdfstring{\cref{row5}}{}}

We apply \cref{prop_Z=Z} to $\cale=\wedge^2 \scQ $ on $G(2,W)$ for $\dim W=5$.
Then we have a subvariety $\Sigma $ in $\P(H^0(\cale) \oplus \wedge^2 W )  = \P(\wedge^2 W \oplus \wedge^2 W)=\P^{19}$.

\begin{lem}\label{lem_def_of_sigma}
It holds that
\[
\Sigma = \left\{ [q,p] \, | \, p \wedge p = q \wedge p=0 \in \wedge^4 W \right\}_{red} \subset \P(\wedge^2 W \oplus \wedge^2 W),
\]
where $X_{red}$ is the scheme with the reduced structure for a scheme $X$.
\end{lem}

\begin{proof}
In this case,
$\scK$ is the kernel of $\wedge^2 W \otimes \calo \rightarrow \wedge^2 \scQ$.
Take a point  $x \in G(2,W) \subset \P(\wedge^{2} W)$ 
and let $p \in  \wedge^{2} W$ be a corresponding element.
Then the fiber $\pi^{-1}(x) $ is $\P( \calk_{x} \oplus \C p)$.
For $q \in \wedge^2 W$,
$q$ is contained in $\calk_{x} \subset \wedge^2 W$ if and only if $q \wedge p= 0$.
Hence it holds that
\[
\Sigma = \left\{  [q, tp] \in \P\left(\wedge^2 W \oplus \wedge^{2} W\right) \, \big| \, [p] \in G(2,W) , t \in \C,  q \wedge p= 0 \right\}.
\]
In particular,
$\Sigma$ is defined by $p \wedge p = q \wedge p =0$ scheme-theoretically outside $ \P\left(\wedge^2 W \oplus \{0\} \right)$.

On the other hand,
$[q, 0] \in  \P\left(\wedge^2 W \oplus \{0\} \right)$ is contained in $\Sigma$
if and only if there exists $[p] \in G(2,W)$ such that $q \wedge p=0$.
This condition is equivalent to $G(2,W) \cap \P(\ker (q \wedge)) \neq \emptyset$,
where $\ker (q \wedge) $ is the kernel of the linear map $q \wedge : \wedge^2 W \arw \wedge^4 W$.
Since $\dim G(2,W) =6$ and 
\[
\dim \ker (q \wedge) \geq \dim \wedge^2 W - \dim  \wedge^4 W =5,
\]
$G(2,W) \cap \P(\ker (q \wedge)) \subset \P(\wedge^2 W )=\P^9$ is non-empty.
Hence $  \P\left(\wedge^2 W \oplus \{0\} \right)$ is contained in $\Sigma$ and we have
\[
\Sigma = \left\{  [q, p] \in \P\left(\wedge^2 W \oplus \wedge^{2} W\right) \, \big| \, p \wedge p=  q \wedge p= 0 \right\} 
\]
as subsets in $\P\left(\wedge^2 W \oplus \wedge^{2} W\right)$.
\end{proof}

For $t \in \C \setminus 0$,
we set 
\begin{align}\label{eq_sigma_t}
\Sigma_t = \left\{  [q, p] \in \P\left(\wedge^2 W \oplus \wedge^{2} W\right) \, \big| \, p \wedge p=  t q  \wedge q + 2 q  \wedge p= 0 \right\}.
\end{align}

\begin{lem}\label{lem_degeneration}
For $t \in \C \setminus 0$,
$\Sigma_t$ is projectively equivalent to the join of two Grassmannians
\begin{align}\label{eq_two_G(2,5)}
G(2,W) \subset \P\left(\wedge^2 W \oplus \{0\} \right), \quad G(2,W) \subset \P\left(  \{0\} \oplus \wedge^2 W\right).
\end{align}
\end{lem}

\begin{proof}
Since 
\[
(tq +p) \wedge (tq +p) = t^2 q  \wedge q + 2t q  \wedge p + p  \wedge p = t ( t q  \wedge q + 2 q  \wedge p) + p \wedge p ,
\]
we have
\begin{align}\label{eq_join}
\left\{  [q, p]  \, \big| \, p \wedge p=  t q  \wedge q + 2 q  \wedge p= 0 \right\}  = \left\{  [q, p]  \, \big| \,  p \wedge p=(tq +p) \wedge (tq +p)= 0 \right\} 
\end{align}
in $ \P\left(\wedge^2 W \oplus \wedge^{2} W\right)$.
The right hand side of (\ref{eq_join}) is projectively equivalent to
\[
\left\{  [q, p]  \, \big| \,  p \wedge p=q \wedge q = 0 \right\}  \subset  \P\left(\wedge^2 W \oplus \wedge^{2} W\right),
\]
which is nothing but the join of the two $G(2,W)$'s in the statement of this lemma.
\end{proof}

Define a closed subscheme $\boldsymbol \Sigma' $ of $\P\left(\wedge^2 W \oplus \wedge^{2} W\right) \times \C$ by
\[
\boldsymbol \Sigma'   = \left\{  ([q, p] ,t) \,  \big| \, p \wedge p=  t q  \wedge q + 2 q  \wedge p= 0 \right\}
\]
and set $\boldsymbol \Sigma =\boldsymbol \Sigma'_{red}$.
By \cref{lem_degeneration}, $\boldsymbol \Sigma $ and $\boldsymbol \Sigma'$ coincide at least over $\C \setminus 0$.
Hence $\Sigma_t $ defined in \cref{eq_sigma_t} for $t \in \C \setminus 0$ is nothing but the fiber of $\boldsymbol \Sigma \rightarrow \C$ over $t \in \C$.

Let $\Sigma_0  \subset \P\left(\wedge^2 W \oplus \wedge^{2} W\right)$ be the fiber of $\boldsymbol \Sigma \rightarrow \C$ over $0 \in \C$.
In this notation,
 \cref{lem_def_of_sigma} states $(\Sigma_0)_{red}=\Sigma$.
In fact, we can show that $\Sigma_0$ is reduced:

\begin{lem}\label{lem_degeneration2}
Under the above notation, $\Sigma_0$ is reduced and
the family $\{\Sigma_t\}_{t \in \C}$ is flat.
In particular,
$\Sigma$ is a flat degeneration of the joins of two Grassmannians.
\end{lem}

\begin{proof}
For $t \neq 0$, $\Sigma_t $ is irreducible of dimension $13$ since it is the join of two Grassmannians $G(2,W)=G(2,5)$.
On the other hand,
$(\Sigma_0)_{red}=\Sigma$ is also irreducible of dimension $13$ by the definition of $\Sigma$ in \cref{eq_def_sigma}.
Since all fibers of the proper morphism $\boldsymbol \Sigma \rightarrow \C$ are irreducible of constant dimension,
the total space $\boldsymbol \Sigma $ is irreducible as well.

Hence, to see the first statement, it suffices to show that $(\Sigma_t)_{red}$ is
normal for any $t \in \C$ by \cite[Section 3, Theorem 9.11]{hartshorne}.
For $t \neq 0$,
the normality of $(\Sigma_t)_{red} = \Sigma_t$ follows from that of $G(2,W)$ and \cref{lem_degeneration}.
For the normality of $(\Sigma_0)_{red}=\Sigma $,
we show the following claim.
\begin{claim} The natural map 
\begin{align}\label{eq_surjectivity}
\Sym^k (\wedge^2 W \oplus \wedge^{2} W)^*  = \Sym^k H^0(G(2,W), \calk^* \oplus \calo(1)) \arw H^0(G(2,W), \Sym^k(\calk^* \oplus \calo(1)))
\end{align}
is surjective for any $k \geq 0$.
\end{claim}
\begin{proof}
It is enough to show the surjectivity of 
\begin{align}\label{eq_tensor}
\Sym^i   H^0(\calk^* ) \otimes \Sym^{j} H^0(\calo(1))  \arw H^0( \Sym^i (\calk^*) (j))
\end{align}
for any $i,j \geq 0$.
The exact sequence $0 \rightarrow \mathcal{Q}(-1) \rightarrow \wedge^2 W^{\ast} \otimes \mathcal{O} \rightarrow \mathcal{K}^{\ast} \rightarrow 0$ induces
  \begin{equation}
\begin{split}\label{eq_reslution}
    0 \rightarrow \mathrm{Sym}^{i-3}(\wedge^2 W^{\ast}) \otimes \mathcal{O}(j-2)  \rightarrow \mathrm{Sym}^{i-2}(\wedge^2 W^{\ast}) \otimes \wedge^2 \mathcal{Q}(j-2) \hspace{30mm}\\
    \rightarrow \mathrm{Sym}^{i-1}(\wedge^2 W^{\ast}) \otimes \mathcal{Q}(j-1)   \rightarrow \mathrm{Sym}^i (\wedge^2 W^{\ast}) \otimes \mathcal{O}(j) \rightarrow  \mathrm{Sym}^i (\mathcal{K}^{\ast})(j) \rightarrow 0.
  \end{split}
    \end{equation}
By the Bott--Borel--Weil theorem and spectral sequence,
it holds that the natural map 
\begin{align}\label{eq_tensor2}
H^0(\mathrm{Sym}^i (\wedge^2 W^{\ast}) \otimes \mathcal{O}(j)) \rightarrow H^0(\mathrm{Sym}^i (\mathcal{K}^{\ast} ) (j))
\end{align}
is surjective.
Since $H^0(\calk^* ) = \wedge^2 W^{\ast}$ and \cref{eq_tensor} factors through \cref{eq_tensor2},
\cref{eq_tensor} is surjective as well.
\end{proof}
By this claim, $\Sigma \subset  \P\left(\wedge^2 W \oplus \wedge^{2} W\right)$ is projectively normal,
and hence normal.
Hence the first statement follows from \cite[Section 3, Theorem 9.11]{hartshorne}.

The last statement follows from \cref{lem_degeneration} and $\Sigma_0=(\Sigma_0)_{red}=\Sigma$.
\end{proof}

The Calabi--Yau $3$-fold that is an intersection of 
two Grassmannians $G(2,5)$ with general positions in $\P^9$ is 
discussed in \cite{kanazawa,miura2,kapustka2}.
By Lemma \ref{lem_degeneration2},
we get the following proposition concerning such Calabi--Yau $3$-folds.

\begin{prop}\label{prop_No4}
Let $Z \subset G(2,W)=G(2,5)$ be the zero locus of a general section of $\wedge^2 \scQ (1)$.
Then $Z$ is a flat degeneration of Calabi--Yau $3$-folds of type $G(2,W) \cap G(2,W) \subset  \P\left(\wedge^2 W \right)$.
\end{prop}

\begin{proof}
Let $s$ be a general section of $\wedge^2 \scQ (1)$ on $G(2,W)$.
Since $H^0(\wedge^2 \scQ) \otimes H^0(\calo(1)) \arw H^0(\wedge^2 \scQ (1))$ is surjective,
we can take a lift $\bar{s} \in H^0(\wedge^2 \scQ) \otimes H^0(\calo(1)) $ of $s$.
Hence we can apply \cref{prop_Z=Z},
and $Z$ is isomorphic to the linear section of $\Sigma$ by $P_{\bar{s}}$.
Since $\Sigma_t $ degenerates to $\Sigma_0=\Sigma$,
$\Sigma_t \cap P_{\bar{s}} $ degenerates to $\Sigma \cap P_{\bar{s}} \simeq Z$.

Let $J \subset \P\left(\wedge^2 W \oplus \wedge^{2} W\right)$ be the join of two Grassmannians in (\ref{eq_two_G(2,5)}).
If a linear subspace $P \subset \P\left(\wedge^2 W \oplus \wedge^{2} W\right)$ of dimension $9= \dim \wedge^{2} W -1$ is general,
$ J \cap P$ is a variety of type $G(2,W) \cap G(2,W) \subset  \P\left(\wedge^2 W \right)$.
More precisely,
$J \cap P= f^{-1} ( G(2,W)) \cap g^{-1}(G(2,W)) \subset P$ holds,
where $ f $ and $g : P \stackrel{\sim}{\rightarrow} \P(\wedge^2 W)$ are isomorphisms obtained as the restrictions of
$ \P\left(\wedge^2 W \oplus \wedge^{2} W\right) \dashrightarrow \P(\wedge^2 W) : [q,p] \mapsto [q]$ and $[q,p]  \mapsto [p]$ respectively.

By Lemma \ref{lem_degeneration}, $\Sigma_t$ is projectively equivalent to the join of two Grassmannians
in $\P\left(\wedge^2 W \oplus \wedge^{2} W\right) $ for $t \neq 0$.
Since $\P(\wedge^2 W) \simeq P_{\bar{s}} $ is general in $\P\left(\wedge^2 W \oplus \wedge^{2} W\right)$,
each $\Sigma_t \cap P_{\bar{s}} $ is a variety of type $G(2,W) \cap G(2,W) \subset  \P\left(\wedge^2 W \right)$ if $t$ is general,
and we obtain this proposition.
\end{proof}

\subsection{\texorpdfstring{\cref{row10}}{}}
\label{ss:9}

Manivel \cite[Theorem 3.1]{Ma} proved that the zero locus of a general section of $\scQ(1)$ on $G(2,n)$
is projectively equivalent to a general linear section of $G(2,n+1)$ of codimension $n$.
We can regard \cref{prop_Z=Z} as a generalization of \cite[Theorem 3.1]{Ma} as follows.

Consider the diagram
\begin{align}\label{eq:diagram}
\xymatrix{
	&  G_{G(2,W)} \left(2, \scS \oplus \calo_{G(2,W)} \right) \ar[ld]_{\mu} \ar[rd]^{\pi} &  \\
	G(2, W \oplus \C)&	 & G(2,W) ,\\
}
\end{align}
where $\pi$ is the Grassmannian bundle and $\mu$ is the morphism induced
by $\scS \oplus \calo_{G(2,W)}  \subset  (W \oplus \C) \otimes \calo_{G(2,W)} $.
Since 
\begin{align*}
G_{G(2,W)} \left(2, \scS \oplus \calo_{G(2,W)} \right)  &= \P(\scS^* \oplus \calo_{G(2,W)} )  \\
&\simeq \P(\scS(1) \oplus \calo_{G(2,W)} ) \simeq  \P(\scS \oplus \calo_{G(2,W)}(-1 )) ,
\end{align*}
this diagram is nothing but the diagram \cref{diagram_pi_mu} for $\cale= \scQ$
and 
\[
\Sigma =G(2, W \oplus \C) \subset \P(\wedge^2(W \oplus \C)) = \P(W \oplus \wedge^2 W).
\]
Hence \cite[Theorem 3.1]{Ma} follows from \cref{prop_Z=Z}.

\vspace{2mm}
This can be also considered as a type-A analogue of \cref{prop_Tits} as follows.
Consider the diagram
\begin{equation}\label{diagram_An}
\begin{gathered}
\xymatrix{
  & 
  F(2,3; \widetilde{W})
  \ar[ld]_{\tilde\mu} \ar[rd]^{\tilde\pi} 
  &	& 
  F(1,3 ; \widetilde{W})
  \ar[ld] \ar[rd] &  \\
  G(2,\widetilde{W})
  & & 
G(3,\widetilde{W})
&  & G(1,\widetilde{W}),
\\
}
\end{gathered}
\end{equation}
where $\widetilde{W} =W \oplus \C$ and  $F(k,k' ; \widetilde{W}) $ is the flag variety parametrizing subspaces $V \subset V' \subset  \widetilde{W}$ with $\dim V=k, \dim V'=k'$.
Then the Tits transform of the point $[\{0\} \oplus \C]  \in G(1,\widetilde{W})$ by the right side is $G(2, W) \subset G(3,\widetilde{W})$,
and the Tits transform of $G(2, W) \subset G(3,\widetilde{W})$ by the left side induces the diagram \cref{eq:diagram},
which corresponds to the diagram \cref{diagram_tits_sub}.


\section{Alternative description: \texorpdfstring{\cref{row16,row21}}{}}\label{sec_description2}

In this section,
we show the following proposition,
which states that Calabi--Yau $3$-folds in \cref{row16} (resp.\  \cref{row21})
are deformation equivalent to general linear sections of $G(2,7)$ of codimension $7$
(resp.\ general linear sections of $G(3,6)$ of codimension $6$).
We note that Calabi--Yau $3$-folds in \cref{row21} are $Z_{\cals^*(1) \oplus \wedge^2 \cals^*}$ in $G(3,6)$,
which can be identified with $Z_{\cals^*(1) \oplus \wedge^2 \calq}$ by \cref{ss:natural}.

\begin{prop}\label{prop_specialize_of_(1)^n}
Let $n=\dim W$
and let $Z \subset G(k,W)$ be the zero locus of a general section of $\scS^* (1) \oplus \wedge^{n-k-1} \scQ$.
Then $Z$ is a flat degeneration of general complete intersections
$Z_{\scO(1)^{\oplus n}} \subset G(k,W)$.
\end{prop}

\begin{proof}
We note that $\wedge^{n-k-1} \scQ$ is isomorphic to $\scQ^* (1)$.
By the exact sequence
\begin{align}\label{eq_dual(1)}
0 \rightarrow \scQ^* (1) \arw W^* (1) \arw \scS^* (1) \arw 0,
\end{align}
we have an exact sequence of global sections
\[
0 \arw H^0( \scQ^* (1)) \arw H^0 (W^* (1))  \stackrel{\varpi}{\arw}  H^0 (\scS^* (1)) \arw 0.
\]

Choose and fix general $(s,q) \in H^0 (\scS^* (1)) \oplus H^0( \scQ^* (1))$.
Since $s$ is general, there exists a general section $\bar{s} \in H^0 (W^* (1)) $
such that $\varpi(\bar{s} ) =s$.
Let $X \subset G(k,W)$ be the zero locus of $s \in H^0 (\scS^* (1))$.
On $X$, we have an exact sequence
\[
0 \arw H^0(X, \scQ^* (1) |_{X}) \arw H^0 (X, W^* (1)|_{X})  \stackrel{\varpi}{\arw}  H^0 (X, \scS^* (1) |_{X}).
\]
Since $\varpi(\bar{s}  |_{X})= s |_{X}=0$,
$\bar{s}  |_{X} \in H^0 (X, W^* (1)|_{X}) $ is contained in $H^0(X, \scQ^* (1) |_{X})$.
By the exact sequence \cref{eq_dual(1)},
the zero locus of $\bar{s} \in  H^0 (W^* (1)) $ in $G(k,W)$ coincides with
the zero locus of $\bar{s}  |_{X} \in H^0(X, \scQ^* (1) |_{X})$ in $X$.

Let $Z_t \subset G(k,W)$ be the zero locus of $q + t \bar{s} \in H^0 (W^* (1))$ for $t \neq 0$.
Since $\varpi (q + t \bar{s}) = t s$, $Z_t$ is contained in $X$.
As a subscheme of $X$,
$Z_t$ is the zero locus of $q + t \bar{s} |_X \in H^0(X, \scQ^* (1) |_X)$.
Since $Z$ is the zero locus of $ q |_X \in H^0(X, \scQ^* (1) |_X)$ as a
subscheme of $X$,
$Z_t$ degenerates to $Z$.

The Koszul complex induced by $(s,q) \in H^0 (\scS^* (1)) \oplus H^0( \scQ^* (1))$ (resp.\ $ q + t \bar{s} \in H^0 (W^* (1))$) gives a locally free resolution of $\calo_{Z}$ (resp.\ $\calo_{Z_t}$) on $G(k,W)$.
Hence $Z$ and $Z_t$ have the same Hilbert polynomial by the exact
sequence \cref{eq_dual(1)}.
Thus this degeneration is flat.
\end{proof}

\begin{rmk}
  By \cref{prop_specialize_of_(1)^n}, we see
  the $G_2$-Grassmannian Calabi--Yau 3-fold $X$ in 
  \cite{G2_CY, annihilator} 
  is a specialization of linear section Calabi--Yau 3-folds 
  $Z_{\scO(1)^{\oplus 7}} \subset G(2,7)$.
\end{rmk}


\section{Alternative description: the rest cases}
\label{sec_description4}
In this section,
we see the rest of descriptions in \cref{tb:cy3} briefly.\\

\cref{row9} : As in the proof of \cite[Proposition 2.1]{Ku},
$Z_{\Sym^2 \scS^*} \subset G(2,6)$ can be identified with the flag variety $F(1,3;\C^4) \subset \P^3 \times \P^3$.
Under the identification,
$\calo(1)|_{Z_{\Sym^2 \scS^*}}$ and $\scS|_{Z_{\Sym^2 \scS^*}}$
correspond to 
\begin{align}
\calo(1,1) :=\calo_{\P^3 \times \P^3}(1,1) |_{F(1,3;\C^4)} \quad \text{and} \quad  (\scS_3/\scS_1) \otimes \scS_1
\end{align}
on $F(1,3;\C^4) $ respectively,
where $\scS_1 \subset \scS_3 \subset \calo_{F(1,3;\C^4)}^{\oplus 4}$ are the universal subbundles of rank $1$ and $3$.
Hence $\scS^*(1)|_{Z_{\Sym^2 \scS^*}}$ corresponds to $ (\scS_3/\scS_1)^* \otimes \scS_1^* \otimes \calo(1,1) = (\scS_3/\scS_1)^*  \otimes \calo(2,1)  $.
Thus the Calabi--Yau $3$-fold $Z_{\Sym^2 \scS^* \oplus \scS^*(1)} \subset G(2,6)$ is
isomorphic to the zero locus of a section of $(\scS_3/\scS_1)^*  \otimes \calo(2,1) $ on $F(1,3;\C^4) $.\\

\cref{row11} : 
$Z_{\wedge^3 \scQ} \subset  G(2,6)$ is a $4$-dimensional Del Pezzo manifold with Picard number two.
By the classification of Del Pezzo manifolds 
(see \cite{IPbook}), it is isomorphic to 
$\P^2 \times \P^2$.\\

\cref{row15}
: It is known that 
$Z_{\wedge^4 \scQ} \subset G(2,7)$ is isomorphic to a 
rational homogeneous space $G_2/P_1$ by
\cite{MR1201387},
(see \cite[Section 16]{CCGK}).\\

\cref{row17,row15,row18} :  
In \cite{Tj},
Tj\o tta studied Calabi--Yau $3$-folds obtained as zero loci of sections of locally free sheaves $ \calo(1)^{\oplus 3}, \calf^* \oplus \calo(2)$, or $\Sym^2 \calf^*$ on the space of determinantal nets of conics $\mathbf{N} $,
where $\calf$ is a locally free sheaf on $\NN$ of rank $2$
(see \cite{EPS}, \cite{Tj} for the definition and properties of $\mathbf{N}$).
We see that these Calabi--Yau $3$-folds are isomorphic to Calabi--Yau $3$-folds in \cref{row17}, \ref{row15}, \ref{row18}, respectively, as follows:


Set $W=\mathfrak{sl}(3) =\mathfrak{sl}(V)$ for $V=\C^3$.
As in \cite[Section 4]{Ku},
we have an $SL(3)$-invariant $ 3$-form $\omega \in \wedge^3 W^*$ by $\omega(X,Y,Z) := \mathrm{Tr} ([X,Y]Z)$.

By the canonical isomorphism $\wedge^3 W^* \simeq \wedge^5 W=H^0(G(2,W),\wedge^5 \calq) $ up to scalar,
we regard $\omega$ as an element in $H^0(G(2,W),\wedge^5 \calq) $
and let $Z \subset G(2,W) $ be the zero locus of $\omega$.
Since $\omega \in \wedge^3 W^* $ is general under $GL(W)$ by \cite[Proposition 4.7]{Ku},
we have $Z=Z_{\wedge^5 \calq} \subset G(2,8)$.
On the other hand,
by the definition of $\omega$,
$Z$ is nothing but the variety of abelian $2$-dimensional subspaces of $W=\mathfrak{sl}(V)$.
In \cite[Section 4]{MR2164624}, it is proved that 
this variety is isomorphic to the space of determinantal nets of conics $\mathbf{N}$.
In fact,
we can check the following:

\begin{lem}\label{lem_F=S}
It holds that $ \iota^* \cals =\calf$ for an isomorphism $\iota : \NN \rightarrow Z_{\wedge^5 \calq} \subset G(2,8)$.
\end{lem}

\begin{proof}
By \cite[Section 3]{Tj}, $SL(V)$ acts on $\NN$, and $\calf $ is an $SL(V)$-equivalent subbundle of $W' \otimes \calo_{\NN}$,
where $W'$ is the kernel of the natural map $\Sym^2 V \otimes V \rightarrow \Sym^3 V$.
Since $W' \simeq \mathfrak{sl}(V) $ as $SL(V)$-modules,
we have an $SL(V)$-equivalent morphism $\iota: \NN \rightarrow G(2,W')=G(2,\mathfrak{sl}(V))$.
For a general point $p \in \NN$,
we have
\[
\calf_{p} =  \{ a yz \otimes x + b xz \otimes y + c xy \otimes z \, | \, a,b,c \in \C, a +b + c =0 \}\subset W'  \subset \Sym^2 V \otimes V 
\]
for a basis $x,y,z $ of $V$
by \cite[Table 1]{Tj}.
Hence $\iota(p) \in G(2,\mathfrak{sl}(V)) $ corresponds to the abelian subspace of
$\mathfrak{sl}(V)$ consists of the diagonal matrices with respect to the basis $x,y,z$.
Since $p \in \NN$ is a general point,
the image $\iota(\NN)$ is contained in $Z=Z_{\wedge^5 \calq} \subset G(2,\mathfrak{sl}(V) )$.

By construction,
$ \iota^* \cals =\calf$ holds.
Since $\iota^* \calo_Z(1) =\det  \iota^* \cals^* = \det \calf^*  $ is ample on $\NN$,
$\iota: \NN \rightarrow Z$ is a finite morphism.
In particular, $\iota(\NN)=Z$ since $\dim \NN=\dim Z=6$.
On the other hand,
the intersection number $(\det \calf^*)^6 $ is $57$ by \cite[Table 2]{Tj},
which coincides with $ {\calo_Z(1)}^6 $.
Thus $\iota$ is a finite morphism of degree one.
Since $Z$ is smooth and hence normal, $\iota : \NN \rightarrow Z$ is an isomorphism.
\end{proof}

By this lemma,
Calabi--Yau $3$-folds obtained from $ \calo(1)^{\oplus 3}, \calf^* \oplus \calo(2), \Sym^2 \calf^*$ on $\NN$
are isomorphic to Calabi--Yau $3$-folds in \cref{row17}, \ref{row15}, \ref{row18}, respectively.
\\

\cref{row22} : We see that $Z_{\Sym^2 \scS^*} \subset G(3,7)$ is an
orthogonal Grassmannian $OG(3,7)\simeq OG(4,8)$.
By the triality of $SO(8)$, it is also isomorphic to $OG(1,8)$,
a quadric hypersurface $Q^6\subset \P^7$, 
which is regarded as the spinor
embedding of $OG(3,7)$. Since
$\cO_{G(3,7)}(1)|_{Z_{\Sym^2 \scS^*}}=\cO_{Q^6}(2)$ under the above
isomorphism,
$Z_{\Sym^2 \scS^*\oplus \scO(1)^{\oplus
3}} \subset G(3,7)$ is nothing but a complete intersection of 
four quadric hypersurfaces in $\P^7$.\\

\cref{row26} :
We recall a result by Reid in \cite{Re}.
\cref{row26} is the case of $k=3$ in the following.

Let $W$ be a vector space of dimension $2k+2$ and let $\C^2 \hookrightarrow \Sym^2 W^*$ be a general pencil of symmetric-forms.
Let $X \subset G(k,W)$ be the zero locus of the section of $(\Sym^2 \scS^*)^{\oplus 2}$ corresponding to this pencil.

Let $l \subset \P(\Sym^2 W^*)$ be the line corresponding to this pencil
and let $D \subset \P(\Sym^2 W^*)$ be the discriminant hypersurface corresponding to degenerate symmetric-forms.
Since the pencil is general,
the line $l  $ intersects with $D$ transversally
and $l \cap D$ consists of $2k+2$ points. 
Let $C \rightarrow l $ be the hyperelliptic curve ramified over $l \cap D$.
Reid proved the following theorem.

\begin{thm}[{\cite[Theorem 4.8]{Re}}]\label{thm_reid}
$X$ is isomorphic as a variety to the Jacobian $J(C)$.
\end{thm}

We note that $\rho(X)=1$ for general $X$.
In fact,
$C$ is general in the space of hyperelliptic curve of genus $k$
since the pencil is general.
Hence the Jacobian $J(C)$ has Picard number one.\\

\cref{row29}: 
Let $\mu : Y \rightarrow \P^5$ be the blowup along the Veronese surface $\nu_2: \P^2 \hookrightarrow \P^5$
and let $ E \subset Y$ be the exceptional divisor.
Kuznetsov \cite[Theorem 4.10]{Ku} showed that $Z_{\wedge^3 \scQ} \subset G(3,8)$ is isomorphic to $Y$.
By this description,
we have the following proposition:

\begin{prop}
$ Z_{\wedge^3 \scQ \oplus \cO(1)^{\oplus 2}} $ in $ G(3,8)$ is isomorphic to a crepant resolution of a complete intersection $Z'$ of two cubic hypersurfaces containing $\nu_2(\P^2)$ in $\P^5$.
The exceptional locus of $ Z_{\wedge^3 \scQ \oplus \cO(1)^{\oplus 2}} \rightarrow Z'$ consists
of a disjoint union of  $30$ $\P^1$'s. 
\end{prop}

\begin{proof}
By \cite{Ku},
$\cO_{ G(3,8)} (1) |_{Z_{\wedge^3 \scQ}} $ corresponds to $\mu^* \cO_{\P^5}(3) \otimes \calo(-E)$ under $Z_{\wedge^3 \scQ} \simeq Y$.
Hence $ Z_{\wedge^3 \scQ \oplus \cO(1)^{\oplus 2}} \subset G(3,8)$ is isomorphic to the complete intersection $Z= D_1 \cap D_2 \subset Y$ of general two divisors $D_1,D_2 \in |\mu^* \cO_{\P^5}(3) \otimes \calo(-E) |$ in $Y$.
Set  $Z' =\mu(Z) \subset \P^5$.
Then $Z' = \mu(D_1) \cap \mu(D_2) $ is a complete intersection of two cubic hypersurfaces containing $\nu_2(\P^2)$.
Since $\mu|_Z :Z \rightarrow Z'$ is birational and $K_Z \sim 0, K_{Z'} \sim 0 $,
$Z$ is a crepant resolution.

The exceptional locus of the resolution $\mu|_Z$
is contained in $E \cap Z= E \cap D_1 \cap D_2$.
Since $\mu : Y \rightarrow \P^5$ is the blowup along $\nu_2(\P^2)$,
$E \rightarrow \nu_2(\P^2)$ is the $\P^2$-bundle $\P_{\nu_2(\P^2)}(N_{\nu_2(\P^2)/\P^5}) \rightarrow \nu_2(\P^2)$,
where $N_{\nu_2(\P^2)/\P^5}$ is the normal bundle of $\nu_2(\P^2) \subset \P^5$.

Take a section $s \in H^0( Y,\mu^* \cO(3) \otimes \calo(-E))$ and let $D \subset Y$ be the divisor defined by $s$.
The restriction of $s$ on $E$ induces a homomorphism $f_s :N_{\nu_2(\P^2)/\P^5}  \rightarrow  \cO_{\P^5}(3) |_{\nu_2(\P^2)}$ on $\nu_2(\P^2) $ by
\begin{align*}
 H^0(E,( \mu^* \cO(3) \otimes \calo(-E)) |_{E}) &\simeq H^0(\nu_2(\P^2) , N^{\vee}_{\nu_2(\P^2)/\P^5} \otimes  \cO_{\P^5}(3) |_{\nu_2(\P^2)}) \\
 & \simeq \Hom_{\nu_2(\P^2)}  (N_{\nu_2(\P^2)/\P^5} , \cO_{\P^5}(3)  |_{\nu_2(\P^2)}) .
\end{align*}
Then the fiber of $E \cap D=\P(N_{\nu_2(\P^2)/\P^5}) \cap D \rightarrow \nu_2(\P^2) $ over $x \in \nu_2(\P^2)$ is the projective subspace of $\P(N_{\nu_2(\P^2)/\P^5 } \otimes k(x))  \simeq \P^2$ corresponding to the kernel of $f_s$ at $x$.

Let $s_1,s_2$ be the section corresponding to $D_1,D_2$
and consider
\[
(f_{s_1},f_{s_2} ): N_{\nu_2(\P^2)/\P^5}  \rightarrow  \cO_{\P^5}(3) |_{\nu_2(\P^2)}^{\oplus 2} .
\]
Then the fiber of $E \cap D_1 \cap D_2\rightarrow \nu_2(\P^2) $ over $x \in \nu_2(\P^2)$ is the projective subspace corresponding to the kernel of $(f_{s_1},f_{s_2} )$ at $x$.
In particular,
the fiber over $x$ is a reduced point, $\P^1$, or $\P^2$ if $(f_{s_1},f_{s_2} )$ has rank $2,1$, or $0$ at $x$, respectively.
Since $D_1,D_2$ are general,
we can compute that the locus where $(f_{s_1},f_{s_2} )$ has rank $1$ (resp.\ $0$) consists of
$30$ points by the Porteous formula (resp.\ is empty).
Thus the exceptional locus of $\mu|_Z :Z \rightarrow Z'$ consists of $30$ $\P^1$'s.
\end{proof}

\cref{row30} : Similarly to \cref{row22},
$Z_{\Sym^2 \scS^*} \subset G(4,8)$ is a
disjoint union of two quadrics $Q^6\simeq OG(4,8)$ and 
the restriction of $\cO_{G(4,8)}(1)$ on each component $Q^6$ coincides with $\cO_{Q^6}(2)$.
Hence $Z_{\Sym^2 \scS^*\oplus \scO(1)^{\oplus
3}} \subset G(4,8)$ is a disjoint union of two
$\left(\P^7\right)_{2^4}$.\\

\cref{row31,row33} : $Z_{(\wedge^2 \scS^*)^{\oplus 2}} \subset G(k,2k)$ is isomorphic to $\prod^k \P^1$
by \cite[Theorem 3.1]{Ku}.
\\

\cref{row32} : As in \cite[Example 4.1]{kuchle}, $Z_{\Sym^2 \scS^* \oplus \wedge^2 \scS^*} $ is of index at least two.
We can also compute $h^{1,1} \ge 4$ for $Z_{\Sym^2 \scS^* \oplus
\wedge^2 \scS^*} $.
Hence $Z_{\Sym^2 \scS^* \oplus \wedge^2 \scS^*} $ is $\P^1 \times \P^1 \times \P^1 \times \P^1$
by the classification of Fano $4$-folds of index two with Picard
number at least two (see \cite{IPbook}).\\


\bibliographystyle{amsalpha}
\bibliography{manuscript}
\end{document}